\documentclass{amsart}
\usepackage{amsmath,amsfonts,amssymb,amscd,amsthm,epsf}
\usepackage{pinlabel}
\newtheorem{prop}{Proposition}[section]
\newtheorem{thm}[prop]{Theorem}
\newtheorem{cor}[prop]{Corollary}
\newtheorem{ques}[prop]{Question}

\newtheorem{conj}[prop]{Conjecture}

\theoremstyle{definition}
\newtheorem{de}[prop]{Definition}

\theoremstyle{remark}
\newtheorem{Remark}[prop]{Remark}             
\newtheorem{Remarks}[prop]{Remarks}             

\def\C{{\mathbb C}}

\def\Z{{\mathbb Z}}
\def\R{{\mathbb R}}
\def\Q{{\mathbb Q}}
\def\F{{\mathbb F}}
\def\inter{\mathop{\rm int}\nolimits}
\def\cl{\mathop{\rm cl}\nolimits}

\def\dim{\mathop{\rm dim}\nolimits}

\def\SO{\mathop{\rm SO}\nolimits}

\def\di{\mathop{\rm div}\nolimits}

\def\co{\colon\thinspace}
\def\J{{\mathcal J}}
\def\s{{\mathcal S}}

\begin{document}
\title{Topological convexity in complex surfaces}
\author{Robert E. Gompf}
\address{
The University of Texas at Austin
}
\email{gompf@math.utexas.edu}
\begin{abstract}
We study a notion of strict pseudoconvexity in the context of topologically (often unsmoothably) embedded 3-manifolds in complex surfaces. Topologically pseudoconvex (TPC) 3-manifolds behave similarly to their smooth analogues, cutting out open domains of holomorphy (Stein surfaces), but they are much more common. We provide tools for constructing TPC embeddings, and show that every closed, oriented 3-manifold $M$ has a TPC embedding in a compact, complex surface (without boundary) realizing any homotopy class of almost-complex structures (the analogue of the homotopy class of the contact plane field in the smooth case). We prove our tool theorems with invariants that classify almost-complex structures on any 4-manifold homotopy equivalent to $M$. These invariants are amenable to computation and respected by homeomorphisms (not necessarily smooth). We study the two equivalence classes of smoothings on the product of a 3-manifold with a line, and on collared ends. Both classes of smoothings are realized by holomorphic embeddings exhibiting any preassigned homotopy class of almost-complex structures. One class arises from TPC embedded 3-manifolds, while the other likely does not.
\end{abstract}
\maketitle


\section{Introduction}

Pseudoconvexity is a fundamental notion in complex analysis. A real hypersurface in a complex manifold is {\em strictly pseudoconvex} (SPC) if it is locally biholomorphic to a strictly geometrically convex hypersurface. As an application, {\em Stein manifolds}, that is, complex manifolds that properly, holomorphically embed in $\C^N$, are characterized as open complex manifolds admitting proper maps to $[0,\infty)$ whose regular level sets are SPC. A basic method for finding Stein manifolds is to locate an SPC hypersurface $M$ in $\C^n$ (or in another Stein manifold). Such an $M$ always cuts out a Stein manifold as the interior of the compact region bounded by $M$. We focus on real dimension 4 (complex {\em surfaces}), where both complex analysis and differential topology are permeated with unique subtleties. In this dimension, topologically embedded 3-manifolds frequently cannot be smoothed by a topological isotopy, and such isotopies typically generate families that cut out open subsets realizing infinitely many diffeomorphism types (which cannot happen in other dimensions). While SPC 3-manifolds are highly constrained, a much more flexible notion of {\em topologically pseudoconvex} (TPC) 3-manifolds was introduced in \cite{steintop}. These are frequently unsmoothable, but still share some basic properties of SPC manifolds. Notably, they cut out open Stein surfaces, although these typically do not have finite differential topology. The present paper investigates TPC 3-manifolds in more depth, exhibiting their full flexibility: The most subtle topological structure of an SPC manifold, namely its contact hyperplane field, has an analogue (at least up to homotopy) for TPC manifolds, and every closed, oriented 3-manifold has a TPC embedding in a closed complex surface, realizing any preassigned choice of this additional structure. We prove this by presenting some basic tool theorems for constructing TPC embeddings and tracking their almost-complex structures. The proofs of these theorems lead us into a deeper discussion of smooth and almost-complex structures on manifolds homeomorphic to $\R\times M^3$.

Our main focus is on 3-manifolds $M$, which we always take to be closed (i.e., compact and without boundary), connected and oriented, and on their topological (i.e., homeomorphic) embeddings in complex surfaces or in more general 4-manifolds $X$. To avoid wild phenomena, we assume the embeddings are {\em bicollared}, that is, they extend in the obvious way to topological embeddings of $\R\times M$. (A {\em collar} is an extension on one side to $[0,\infty)\times M$.) Such an embedding of $M$ frequently cannot be smoothed by a topological isotopy (homotopy through embeddings), and the smooth structure on $\R\times M$ pulled back by such an embedding is typically exotic (not diffeomorphic to the standard smoothing). A topological isotopy of an embedding of $\R\times M$ (or other open 4-manifold) may induce a 1-parameter family of pairwise nondiffeomorphic smoothings on the domain. (Such nonuniqueness can only occur in dimension 4 since the smoothings are sliced concordant, as we discuss below.) We will frequently deal with a 3-manifold $M$ exhibited as the bicollared boundary of a topologically embedded 4-manifold $Y$ in a complex surface. Following \cite{steintop}, we will call such a $Y$ {\em topologically pseudoconvex (TPC)} if it is also a {\em Stein compact}. The latter is a standard term for a compact subset with a Stein neighborhood system, so every neighborhood of it contains a neighborhood that is Stein (in the inherited complex structure). By a classical result, the interior of a Stein compact is always Stein. As in \cite{steintop}, an embedding of a 3-manifold $M$ will be called {\em TPC} if it has a neighborhood biholomorphic to that of a boundary of a TPC 4-manifold (embedded in a possibly different complex surface). A TPC immersion is defined similarly using a holomorphic immersion of the boundary neighborhood. TPC embeddings and immersions behave analogously to their SPC counterparts: A TPC immersion determines an orientation on $M$, which we assume agrees with its preassigned orientation, and a TPC embedding in a Stein surface cuts out a TPC 4-manifold, whose interior is then Stein. However, TPC embeddings are much more flexible than SPC embeddings. It is shown in \cite{steintop} that every topological embedding of a 2-handlebody (a 4-dimensional handlebody with all handles of index at most 2) is topologically isotopic to a TPC embedding. From this, it is deduced that every closed 3-manifold $M$ has a TPC immersion in $\C^2$ and embedding in every closed, simply connected complex surface with $b_\pm$ sufficiently large. The present paper explores the range of additional structure realized by such immersions and embeddings.

The most subtle topological structure associated to an SPC 3-manifold $M\subset X$ is its induced contact plane field. This is given by $\xi=TM\cap JTM$, where $J$ represents fiberwise multiplication by $i$ in the tangent complex 2-plane bundle $TX$. While it is not clear how much of the classification theory of contact structures usefully extends to TPC embeddings, we can at least capture the homotopy class of the plane field $\xi$. For an SPC embedding, $M$ has a neighborhood diffeomorphic to $\R\times M$. The above formula for $\xi$ induces a bijection from homotopy classes of almost-complex structures $J$ on $\R\times M$ to those of oriented plane fields on $M$. (The inverse interprets $\xi$ and its complementary trivial $\R^2$-bundle as complex line fields determining $J$ up to homotopy on $\R\times M$.) A TPC 3-manifold $M$ has a neighborhood $V$ {\em homeomorphic} to $\R\times M$, but typically with an exotic smooth structure, inheriting a complex structure $J$ from the ambient complex surface. We will find a way to keep track of the homotopy class of such an almost-complex structure on the underlying topological manifold $\R\times M$, and show that such structures are much more flexible than their analogs for SPC surfaces.  Theorem~\ref{main} gives conditions guaranteeing TPC embeddings and immersions that preserve such structure. Some sample consequences are given below in Theorem~\ref{main0}. For $\gamma$ in a finitely generated abelian group $G$, let $\di \gamma\in\Z^{\ge0}$ denote the divisibility of $\gamma$ in $G$ modulo torsion, with $\di\gamma=0$ when $\gamma$ has finite order. For $G=H^2(M)$ and $J$ on $V$ as above, we say that $k\in\Z$ is a {\em factor} of the Chern class $c_1(J)$ if $c_1(J)=k\beta$ for some $\beta\in H^2(M)$. (Note that if $c_1(J)$ has order $n$ then all integers congruent to 1 mod $n$ are factors.) Since a TPC immersion $f\co M\to X$ satisfies $f^*c_1(X)=c_1(J)$, $\di c_1(X)$ must be a factor of $c_1(J)$ whenever $f$ exists and $H^2(X)$ is finitely generated and torsion free.

\begin{thm}\label{main0} Let $(M,J)$ be a closed, oriented 3-manifold with an almost-complex structure (up to homotopy) on $\R\times M$. Then:
\begin{itemize}
\item[a)] $(M,J)$ has a TPC embedding in every closed, simply connected complex surface $X$ with $\di c_1(X)$ a factor of $c_1(J)$ and $b_\pm(X)$ sufficiently large (relative to a preassigned upper bound on $\di c_1(X)$ if $c_1(J)$ has finite order).
\item[b)] $(M,J)$ has a TPC immersion in a preassigned closed, simply connected complex surface $X$ if and only if $\di c_1(X)$ is a factor of $c_1(J)$.
\item[c)] $(M,J)$ has a TPC immersion in every nonminimal or ruled surface.
\item[d)] $(M,J)$ has a TPC immersion in a preassigned complex surface $X$ with $c_1(X)=0$ (such as $\C^2$) if and only if $c_1(J)=0$.
\end{itemize}
\end{thm}

\noindent In particular, every homotopy class of structures $J$ on $M$ is realized by TPC embeddings as in (a), since every nonminimal $X$ with $b_\pm(X)$ sufficiently large satisfies its hypotheses with $\di c_1(X)=1$. (The integers $b_\pm(X)$ are the dimensions, respectively, of the maximal positive and negative definite subspaces of the rational intersection form.)  The topology of the embeddings is flexible. As constructed in the proof of Theorem~\ref{main}, the embeddings of a given $(M,J)$ cut manifolds $X$ into two pieces, each with $b_\pm$ arbitrarily large, and the pieces with oriented boundary $M$ are minimal (in fact, spin). In constrast, while every closed 3-manifold has infinitely many homotopy classes of oriented plane fields, each of which is realized by at least one contact structure, most such homotopy classes cannot arise for any SPC embedding cutting out a compact region. For example, the manifolds $S^3$, $S^1\times S^2$ and connected sums of these each admit a unique contact plane field realizable in this way, and the diffeomorphism type of the enclosed region is uniquely determined up to blowups ($B^4$, $S^1\times B^3$ and boundary sums of these, respectively). (See \cite{CE}.) A more drastic example is the connected sum $P\#\overline{P}$, where $P$ is the Poincar\'e homology sphere and the bar denotes orientation reversal throughout the text. This admits no tight contact structure \cite{EH}, so no SPC embedding (or immersion), but it has a TPC embedding in $\C^2$ (as the boundary of the 2-handlebody $I\times(P-\inter B^3)$). The resulting almost-complex structure is unique, but all structures on $P\# \overline{P}$ are realized by TPC immersions in $\C^2$ and embeddings as in (a).

The almost-complex structures are surprisingly delicate to track. We work out the details in Section~\ref{Conc} (which is independent of Section~\ref{TPC}). For a 4-manifold $V$ homotopy equivalent to $M$, there is a bijection from the set $\J(V)$ of homotopy classes of almost-complex structures on $V$ to homotopy classes of continuous maps $[V,S^2]\cong[M,S^2]$ (cf.\ \cite[Remarks~4.1.10,~4.1.12]{McS} for closed 4-manifolds). Maps from a 3-complex to $S^2$ were classified up to homotopy by Pontryagin \cite{P} in 1941 using obstruction theory. However, the first bijection is not canonical. Furthermore, proving the above theorem requires us to compare classes in $\J(\R\times M)$ induced by different maps to complex surfaces, which is awkward from the obstruction theory viewpoint. For example, we must keep track of these classes during topological isotopies of $\R\times M$ in $X$, which will typically change the smooth structure on $\R\times M$. While the isotopy will provide a bijection of homotopy classes, it is not a priori clear that the bijection is independent of the choice of isotopy. To remedy these difficulties, we need invariants that classify homotopy classes of almost-complex structures and are isotopy-invariant and directly computable. Fortunately, analogous invariants for plane fields $\xi$ on 3-manifolds were constructed in \cite{Ann}, and these can be adapted. Pontryagin's primary uniqueness obstruction is detected by a surjection $\Gamma$ to $H^2(M)$, depending on a choice of spin structure $s$ (with $\Gamma(\xi,s)$ the difference class between the spin$^\C$-structures determined by $\xi$ and $s$). This satisfies the condition $2\Gamma(\xi,s)=c_1(\xi)$ but is delicate in the presence of 2-torsion. Pontryagin's secondary obstruction gives differences in $\Z/\di c_1(\xi)$ for classes with the same value of $\Gamma$. We measure this by an invariant $\tilde\Theta$ coming from the equality $c_1^2(X)=2\chi(X)+3\sigma(X)$ for closed complex surfaces and its failure when $\partial X=M\ne\emptyset$. (Throughout the text, $\chi$ and $\sigma$ denote the topological Euler characteristic and signature.) The invariants $\Gamma$ and $\tilde{\Theta}$ have become standard tools in contact topology for contact structures for which $c_1(\xi)$ has finite order (allowing a simplified version $\theta$ of $\tilde{\Theta}$). However, we also need them when $c_1(\xi)$ has infinite order. In this case, $\tilde\Theta$ depends on both $s$ and delicate framing data. The present paper probably represents its first essential use in full generality. The invariants $\tilde{\Theta}$ and $\theta$ also contain unexplored information beyond Pontryagin's secondary obstruction (Remark~\ref{tau}(b)).

To understand the effect of topological isotopy on almost-complex structures, it is convenient to pass to a weaker, intrinsic equivalence relation from classical smoothing theory. A {\em sliced concordance} between two smoothings of an open topological manifold $V$ is a smoothing of $I\times V$ restricting to the given smoothings on the boundary components, such that projection to $I=[0,1]$ is a smooth submersion. The levels $V_t$, $t\in I$, then comprise a 1-parameter family of smoothings of $V$, i.e.,\ smooth manifolds homeomorphically identified with $V$. An isotopy of a codimension-0 topological embedding $V\to X$ can be interpreted as a level-preserving topological embedding $I\times V\to I\times X$, and so pulls back the product smoothing to a sliced concordance. When $\dim V\ne 4$, sliced concordant smoothings are related by a diffeomorphism (that is $C^0$-small isotopic to the identity through homeomorphisms \cite{KS}); in particular, topologically isotopic, codimension-0 embeddings pull back diffeomorphic smoothings. However, in dimension 4, such isotopies frequently realize uncountably many diffeomorphism types. For general open 4-manifolds, sliced concordance classes are classified by $H^3(V;\Z_2)$, which is $\Z_2$ in our case of interest \cite{KS}. (These are also the same as stable isotopy classes, where we allow isotopy through homeomorphisms after product with $\R$.) A sliced concordance identifies the sets $\J(V_0)$ and $\J(V_1)$. (Use the bundle of tangents to the slices $V_t$ rather than stabilizing to a 6-dimensional bundle, which would lose information.) For $V$ homotopy equivalent to $M$, we show that the invariants $\Gamma$ and $\tilde{\Theta}$ are preserved by sliced concordance (Theorem~\ref{class}), so that the identification is independent of the choice of sliced concordance (Corollary~\ref{sliced}), and hence, of topological isotopy. While $\J$ is clearly functorial for diffeomorphisms, we ultimately conclude that it is functorial for homeomorphisms, at least for manifolds homeomorphic to $\R\times M$ (Theorem~\ref{Jfunctor}). The proofs exploit a canonical $\Z$-action on $\J(V)$ that respects all of the above structure.

We examine sliced concordance classes of smoothings of $\R\times M$ in more detail in Section~\ref{Ends}. The Kirby--Siebenmann uniqueness obstruction distinguishes the two sliced concordance classes, but is not in general a canonical map to $\Z/2$, merely a torsor. However, for many 3-manifolds, the two classes are canonically distinguished by the existence of a product smoothing, and no diffeomorphism relates smoothings in different classes:

\begin{thm}\label{RxM}
If $M$ is hyperbolic, Haken or a $\Z/2$-homology sphere, then diffeomorphic smoothings of $\R\times  M$ must be sliced concordant. In the first two cases, any smoothing diffeomorphic to some $\R\times M'$ is sliced concordant to the standard smoothing of $\R\times  M$.
\end{thm}

\noindent We prove this after Proposition~\ref{muplus8}. The simplest 3-manifolds violating these hypotheses are the even lens spaces; we show (Proposition~\ref{lens}) that most of these also satisfy the first conclusion. Then we prove an application to more general 4-manifold smoothing theory.

\begin{cor}\label{end}
Suppose an end of a 4-manifold is homeomorphic to $\R\times M$ with $M$ hyperbolic, Haken or $S^3$. Then there is a canonical $\Z/2$-invariant distinguishing the two sliced concordance classes of smoothings of the end, with value 0 on each product smoothing. If $M$ is a $\Z/2$-homology sphere, such a $\Z/2$-invariant is established by fixing the choice of $M$ (among 3-manifolds $M'$ for which the end is homeomorphic to $\R\times M'$).
\end{cor}

\noindent The author knows no examples of 3-manifolds for which these conclusions fail. However, the first conclusion of the theorem fails for more general 4-manifolds. For example, while it holds for a twice-punctured 4-sphere $\R\times S^3$, it fails after a third puncture: If $V$ is taken from the nonstandard sliced concordance class on $\R\times S^3$, then $V-\{ p\}$ has a unique end admitting a smooth product structure (distinguished by the above invariant). But homeomorphisms act transitively on the three ends, pulling back the given smoothing into different sliced concordance classes. Combining these results with previous sections, we show (Theorem~\ref{holom}, cf.\ Corollary~\ref{2hCplx}) that every sliced concordance class and almost-complex structure on $\R\times M$ are realized holomorphically together, by an embedding into a closed, complex surface. For the class of the standard smoothing, the embedding arises as a boundary collar of a TPC 2-handlebody. For the other class, the  embedded $M$ cannot even cut out a smoothable 4-manifold. It is an interesting question whether it can ever be TPC. Other related questions and a conjecture also appear in Section~\ref{Ends}.

Except where otherwise specified, we use the following conventions: Manifolds are smooth and oriented, and local homeomorphisms and diffeomorphisms preserve orientation. Handlebodies are compact. Homology and cohomology have integer coefficients. For any 4-manifold $V$, $\J(V)$ denotes the set of homotopy classes of almost-complex structures on $V$ (implicitly respecting the given orientation on $V$). We canonically identify $\J(\R\times M)$ with $\J(M)$, the set of homotopy classes of oriented plane fields on $M$.

\section{Constructing TPC maps}\label{TPC}

In this section, we construct TPC embeddings and immersions of any (closed, oriented) 3-manifold $M$, realizing each $J\in\J(M)$. We use a simply stated TPC embedding theorem from \cite{steintop}, along with several tool theorems for tracking almost-complex structures that we prove in Section~\ref{Conc}.

\subsection{TPC embeddings} 

We begin with the embedding theorem from \cite{steintop}.

\begin{de}\label{TPCdef}
 A compact topological 4-manifold topologically embedded with bicollared boundary in a complex surface is {\em topologically pseudoconvex (TPC)} if it has a Stein neighborhood system. A topological embedding (resp.\ immersion) $f\co M\to X$ of a 3-manifold into a complex surface will be called {\em TPC} if there is a TPC 4-manifold $Y$ in some complex surface $X'$ and a homeomorphic identification of $M$ with $\partial Y$ so that $f$ extends to a holomorphic embedding (resp.\ immersion) of some neighborhood of $\partial Y$ in $X'$.
 \end{de}
 
 The TPC 4-manifolds that we construct will always come endowed with the following additional structure. Let $\Sigma$ denote the intersection of the standard Cantor set with the open interval $(0,1)$.

\begin{de}\label{oniondef}
A {\em Stein onion} consists of a closed 3-manifold $M$, a Stein surface $U$ and a continuous surjection $\psi\co[0,1) \times M \to U$ restricting to a homeomorphic embedding on $(0,1)\times M$ such that for each $\sigma\in\Sigma$, the open  subset $\psi([0,\sigma) \times M)$ is Stein.
\end{de}

\noindent Thus, $\psi$ presents $U$ as an open mapping cylinder exhibiting the {\em core} $\psi(0)$ as a deformation retract of $U$, so the core has the homotopy type of a 2-complex. Consider the uncountable subset $\Sigma'\subset\Sigma$ obtained by deleting the countably many boundaries of open intervals comprising $(0,1)-\Sigma$. For each $\sigma\in\Sigma'$, the subset $Y_\sigma=\psi([0,\sigma]\times M)$ is a TPC 4-manifold that is a nested intersection of uncountably many topologically isotopic Stein surfaces of the form $\inter Y_\tau$ with $\tau\in\Sigma$. (A similar notion of convexity applies to the core. In particular, the core and each such $Y_\sigma$ are Stein compacts.) Furthermore, $\inter Y_\sigma$ is an uncountable nested union of topologically isotopic TPC 4-manifolds. The levels $\partial Y_\sigma$ with $\sigma\in\Sigma'$ comprise an uncountable, topologically parallel collection of TPC embeddings of $M$. This definition of a Stein onion is weaker than the one in \cite{steintop} (compact case), which presented the closure of $U$ in a complex surface as a topologically embedded 2-handlebody, with its core 2-complex given by $\psi(0)$ and smoothly embedded except for one point on each 2-cell (where it is topologically tame but not smoothly conelike). The main theorem of \cite{steintop} gives a simple and powerful method for locating Stein onions in complex surfaces:

\begin{thm}\label{onion}\cite[Theorem~1.4]{steintop}
Every topological embedding of a 2-handlebody $Y$ into a complex surface is topologically isotopic to an embedding with image $Y_\sigma$ in some holomorphically embedded Stein onion.
\end{thm}

\noindent The proof combines fundamental work of Eliashberg \cite{CE} and Freedman \cite{F}, \cite{FQ} to isotope $Y$ to the closure of a Stein onion with $\psi$ respecting the handle structure, so that the 2-handles of each $Y_\sigma$ with $\sigma\in\Sigma$ have well-controlled (infinite) differential topology (as generalized Casson handles). Such a Stein onion can typically be arranged to realize infinitely many (often uncountably many) diffeomorphism types of Stein surfaces $\inter Y_\sigma$, $\sigma\in\Sigma$ \cite[Section~5]{steintop}. All of our TPC 4-manifolds arise from this theorem, so automatically inherit all of the above structure, and so our TPC 3-manifolds occur in uncountable, topologically parallel families exhibited by bicollars.

\subsection{TPC maps realizing a fixed almost-complex structure}

We next wish to state the main theorem of Section~\ref{TPC} and derive Theorem~\ref{main0}. This requires some definitions and a theorem from Section~\ref{Conc}. For every bicollared topological embedding of $M$ into a smooth 4-manifold $X$, the bicollar $\R\times M\hookrightarrow X$ is unique up to topological isotopy \cite{FQ}, so the induced smoothing on $\R\times M$ is unique up to sliced concordance. The same applies to topological immersions if we assume an immersed bicollar and isotope in its domain. Every topological 4-manifold $Y$ with boundary $M$ has uniquely collared boundary, determining a unique isotopy class of bicollared embeddings $M\hookrightarrow\inter Y$. Any smoothing on $\inter Y$ then determines a unique sliced concordance class of smoothings on $\R\times M$, which we call {\em boundary collar smoothings}. For example, if $Y$ is created by Theorem~\ref{onion}, $\partial Y$ may not be smoothable by any isotopy, but $\inter Y$ inherits a smoothing. This induces boundary collar smoothings, defined by pulling the smoothing of $\inter Y$ back by $\psi$ (Definition~\ref{oniondef}) to some $(a,b)\times M$ (identified with $\R\times M$ by some orientation-preserving homeomorphism $\R\approx (a,b)$). If $X$ is a complex surface, every smoothing $V$ on $\R\times M$ induced by a bicollared immersion inherits a complex structure. We will show (Theorem~\ref{Jfunctor}) that $\J$ is functorial with respect to homeomorphisms, so $\J(V)\cong\J(\R\times M)$ is well-defined on the underlying topological manifold. For now, the following special case of Corollary~\ref{sliced} suffices:

\begin{thm}\label{sliced0}
The set $\J(V)$ is well-defined on sliced concordance classes of smoothings $V$ of $\R\times M$.
\end{thm}

\begin{de}
For $J\in\J(M)=\J(\R\times M)$, a {\em topological embedding} (resp.\ {\em immersion}) of the pair $(M,J)$ in a complex surface is a bicollared topological embedding (resp.\ immersion) of $M$ for which the induced smoothing $V$ of $\R\times M$ is sliced concordant to the standard smoothing, with $J\in\J(\R\times M)\cong\J(V)$ the inherited complex structure on $V$. It will be called {\em TPC} if the map on $M$ is TPC.
\end{de}

The main theorem of Section~\ref{TPC} generalizes Theorem~\ref{main0}, giving broad conditions under which a given pair $(M,J)$ has a TPC embedding or immersion.

\begin{thm}\label{main} Given a closed, oriented 3-manifold $M$ and $J\in\J(M)$,
\begin{itemize}
\item[a)] $(M,J)$ has a TPC immersion in a given complex surface $X$ whenever there is an $\alpha\in\pi_2(X)$ for which $\langle c_1(X),\alpha\rangle$ is a factor of $c_1(J)$.
\item[b)] There is an integer $n_M$ depending only on $M$ such that $(M,J)$ has a TPC embedding in every closed, simply connected complex surface $X$ with $b_\pm(X)\ge n_M+2(\di c_1(X))^2$ and $\di c_1(X)$ a factor of $c_1(J)$. If $X$ is spin, the inequality can be weakened to $b_\pm(X)\ge n_M+\frac12(\di c_1(X))^2$.
\end{itemize}
Let $m$ denote $\langle c_1(X),\alpha\rangle$ in (a) or $\di c_1(X)$ in (b). If $X$ is spin, then the resulting immersion or embedding can be chosen to realize any given spin structure on $M$, unless $H_1(M)$ has 2-torsion and $m$ is divisible by 4.
\end{thm}

\noindent In particular, for fixed $M$, every homotopy class $J$ has a TPC embedding in every nonminimal, closed, simply connected $X$ with $b_\pm(X)\ge n_M+2$. For minimal examples satisfying the inequalities in (b), consider the hypersurface $X_d$ of degree $d$ in $\C P^3$. This has $\di c_1(X_d)=|d-4|$, but the numbers $b_\pm(X_d)$ increase cubically with $d$. It is spin whenever $d$ is even. Complete intersections in $\C P^n$ provide other such families. For a given $M$, the proof gives a computable $n_M$. The divisibility hypothesis in (b) is clearly necessary. The dependence of the bounds in (b) on $\di c_1(X)$ and the final caveat on 2-torsion and $m$ for both (a) and (b) are also necesssary; see Remark~\ref{necHyps}(a). Note that it makes sense to pull a spin structure back to $M$, since spin structures can be interpreted on their underlying topological manifolds (in this case $\R\times M$), cf.~Theorem~\ref{spinc}.

\begin{proof}[Proof of Theorem~\ref{main0}]
Part (a) follows immediately from (b) above since $c_1(J)$ has only finitely many factors when it has infinite order. The ``only if" direction of (b) is easy. For the converse, note that $H^2(X)$ is torsion-free, so we can factor $c_1(X)$ as $(\di c_1(X))\beta$ for a primitive class $\beta$. By Poincar\'e duality, there is a dual class $\alpha\in H_2(X)$ with $\langle\beta,\alpha\rangle=1$. Since $X$ is simply connected, $\pi_2(X)$ can be identified with $H_2(X)$, so we can apply (a) above. To prove (c), note that a nonminimal or ruled surface $X$ contains a holomorphically embedded sphere $\alpha$ of square $-1$ or 0, so $\langle c_1(X),\alpha\rangle=1$ or 2 by the adjunction formula. (It would also suffice for $X$ to have a holomorphic sphere of square $-3$ or $-4$.) Since $c_1(J)$ reduces mod 2 to $w_2(M)=0$, it has 2 as a factor. Part (d) follows immediately from (a) above by letting $\alpha=0$ (or any other class).
\end{proof}

\subsection{Proof of Theorem~\ref{main}}
 The proof requires two additional theorems that, along with Theorem~\ref{sliced0}, are proved after Corollary~\ref{sliced} from stronger results in Section~\ref{Conc}. First, we need to understand the set $\J(M)$. This was studied in \cite{Ann} in the context of plane fields, and classified by two invariants:

\begin{thm}\label{Ann}\cite{Ann}
There is a canonical $\Z$-action on $\J(M)$, with structure determined by two invariants:
\begin{itemize}
\item[(a)] The set of orbits is identified with $H^2(M;\Z)$ by an invariant $\Gamma$ that depends on a choice of spin structure $s$ on $M$. For fixed $J$, $\Gamma(J,s)$ ranges over all classes with $2\Gamma(J,s)=c_1(J)$.
\item[(b)] The orbit containing a given $J$ is isomorphic as a $\Z$-space to $\Z/\di c_1(J)$. There is an invariant $\tilde\Theta$ that, for a fixed $s$ and (normally) framed 1-manifold Poincar\'e dual to $\Gamma(J,s)$, equivariantly identifies this orbit with an index-4 coset of $\Z/4\di c_1(J)$.
\end{itemize}
\end{thm}

\noindent We will prove this theorem in more generality in Section~\ref{Conc}, by defining the invariants for $\J(V)$ whenever the 4-manifold $V$ is homotopy equivalent to $M$. Dependence of the invariants on the auxiliary data is given in \cite{Ann} and Proposition~\ref{vary}. Notably, the chosen framing in (b) is well defined mod $\di c_1(J)$ on the homology class dual to $\Gamma(J,s)$ (cf.\ Proposition~\ref{orbit}).

 Our main tool for applying these invariants to TPC embeddings is the following theorem, which we will apply to the output of Theorem~\ref{onion}. That is, we will exhibit the required $M$ in $X$ as the boundary of a topologically embedded (or immersed) $Y$ that, by construction, will be a spin 2-handlebody without 1-handles. Then $\inter Y$ will inherit a smoothing and almost-complex structure. We wish to represent $\Gamma$ by a framed 1-manifold in $M$, then compute $\tilde{\Theta}$ via the obstruction to extending this framing over a suitable surface in $Y$.

\begin{thm}\label{Y}
Let $(Y,s)$ be a smooth, compact, spin 4-manifold with boundary $M$, $H^3(Y;\Z/2)=0$ and no 2-torsion in $H^2(Y;\Z)$. On $\inter Y$, fix another smoothing and an almost-complex structure $J_Y$. Then any induced boundary collar smoothing $V$ is sliced concordant to the standard smoothing on $\R\times M$, identifying $\J(V)$ with $\J(M)$. Let $(F,\partial F)\subset (Y,M)$ be a compact surface with a (normal) framing $\phi$ on $\partial F\subset M$, such that $2[F]\in H_2(Y,M)$ is Poincar\'e dual to $c_1(J_Y)$. Then the restriction $J_V=J_Y|V$ has $\Gamma(J_V,s)$ Poincar\'e dual to $[\partial F]$ and $\tilde{\Theta}(J_V,s,\phi)\equiv 4e(\nu F,\phi)-2\chi(Y)-3\sigma(Y)$ mod $4\di c_1(J)$, where $e(\nu F,\phi)$ is the normal Euler number of $F$ relative to $\phi$.
\end{thm}

\noindent Since $Y$ is spin, 2 is a factor of $c_1(J_Y)$, so a suitable $F$ always exists. Note that when $\partial F$ is empty, $4e(\nu F,\phi)=c_1^2(J_Y)$ (which is well-defined in this case). Thus, $\tilde{\Theta}(J_V,s,\phi)$ is a relative version of $c_1^2(J_X)-2\chi(X)-3\sigma(X)$, which vanishes for closed, almost-complex 4-manifolds $X$.

\begin{proof}[Proof of Theorem~\ref{main}]
First we show that, for $m$ as given, there is a spin structure $s$ on $M$ such that either $m$ (when it is odd) or $\frac{m}{2}$ is a factor of $\Gamma(J,s)$, and that every spin structure $s$ has this property unless $H_1(M)$ has 2-torsion and $m$ is divisible by 4. In both cases (a, b) of the theorem,  $m$ is a factor of $c_1(J)=2\Gamma(J,s)$ for every $s$. If $m$ is divisible by 4, we can choose $s$ so that $\frac{m}{2}$ is a factor of $\Gamma(J,s)$. (Write $c_1(J)=2(\frac{m}{2}\beta)$ and apply the last sentence of Theorem~\ref{Ann}(a)). This covers all spin structures when $H_1(M)\cong H^2(M)$ has no 2-torsion, since $\Gamma(J,s)$ is then independent of $s$. If $m$ is not divisible by 4, either $m$ or $\frac{m}{2}$ is an odd integer $2r+1$. In the quotient $H^2(M)/(2r+1)H^2(M)$, we have $2\Gamma(J,s)=0$ for each $s$ so $\Gamma(J,s)=(2r+1)\Gamma(J,s)=0$. Thus $\Gamma(J,s)$ has the required factor.

We can now describe $M$ as a suitable boundary. It is well known (e.g.\ \cite{GS}) that every spin 3-manifold bounds a spin 2-handlebody without 1-handles. Choose such a pair $(Y,s_Y)$ bounded by $(M,s)$. Since $H^3(Y,M)\cong H_1(Y)=0$, the restriction map $H^2(Y)\to H^2(M)$ is surjective. Thus, $\Gamma(J,s)$ pulls back to a class $\gamma\in H^2(Y)$. By factoring $\Gamma(J,s)$ as in the previous paragraph before pulling back, we can arrange $m$ to be a factor of $2\gamma$. Choose a compact surface $F$ in $Y$ Poincar\'e dual to $\gamma$, and a framing $\phi$ on $\partial F\subset M$, and let $\theta=4e(\nu F,\phi)-2\chi(Y)-3\sigma(Y)\in\Z/4\di c_1(J)$.

To prove (a), note that we can modify $Y$ to change $\theta$ by any multiple of 4, for example by taking the connected sum (away from $F$) with copies of $S^2\times S^2$ or the K3-surface. The modified $Y$ is still a 4-ball with 2-handles attached, so it is homotopy equivalent to a wedge of 2-spheres. For the given class $\alpha\in\pi_2(X)$, we have arranged that $\langle c_1(X),\alpha\rangle=m$ is a factor of $2\gamma$. Thus, we can construct a map $f\co Y\to X$ such that $f^*c_1(X)$ agrees with $2\gamma$ when evaluated on $H_2(Y)$, by sending each sphere of the wedge to a suitable multiple of $\alpha$. Since $H_1(Y)=0$, $H^2(Y)$ equals ${\rm Hom}(H_2(Y),\Z)$, so $f^*c_1(X)=2\gamma$. In particular, $f^*w_2(X)=0=w_2(Y)$. We can assume $f$ is an embedding on the 4-ball and an immersion on the cores of the 2-handles. After adding double points to the images of these cores, we can assume their normal Euler numbers relative to fixed framings in $\partial B^4$ are the same as in $Y$. (They are initially the same mod 2 by the $w_2$ condition, and each double point changes the difference by $\pm2$.) We can then assume $f$ is an immersion. Let $J_Y=f^*J_X$ be the pulled back complex structure on $Y$. Apply Theorem~\ref{onion} to $Y$ inside itself, isotoping $Y$ to $Y_\sigma$ inside a Stein onion in $Y$, and let $V$ be a boundary collar smoothing in $Y_\sigma$ that exhibits TPC embeddings of $M$ inside $Y$. Applying $f$ to these gives TPC immersions of $M$ in $X$ since $f\co Y\to X$ is holomorphic by the definition of $J_Y$. Since $c_1(J_Y)=f^*c_1(X)=2\gamma$ is Poincar\'e dual to $2[F]$, Theorem~\ref{Y} implies that $\Gamma(J_V,s)=\gamma|M=\Gamma(J,s)$ (where we canonically identify spin structures on $V$ with those on $M$). Thus, by Theorem~\ref{Ann}(a), $J_V$ lies in the same $\Z$-orbit as $J$. By Theorem~\ref{Y}, $\tilde{\Theta}(J_V,s,\phi)=\theta$, which we have already seen can take any preassigned value in the appropriate coset. Thus, we can assume $\tilde{\Theta}(J_V,s,\phi)=\tilde{\Theta}(J,s,\phi)$, so $J_V=J$ in $\J(V)=\J(M)$ by Theorem~\ref{Ann}(b), proving (a). Since $H^1(Y;\Z/2)=0$, $s_Y$ is the unique spin structure on $Y$, so any spin structure on $X$ pulls back to $s_Y|V$ on $V$, which we have identified with $s$ on $M$. This proves the last sentence of the theorem in the case of immersions. Since the proof applies to any $(M,J)$ (e.g.~ by taking $X$ nonminimal so $m=1$), we have also shown that any $\theta$ as constructed in the second paragraph is congruent mod 4 to every $\tilde{\Theta}(J',s,\phi)$ for which $\Gamma(J',s)=\Gamma(J,s)$ (cf.~Proposition~\ref{mod4}). Furthermore:

\begin{cor}\label{2hCplx}
Every $(M,J,s)$ bounds a spin complex surface $Y$ made by adding 2-handles to $B^4$. \qed
\end{cor}

To prove (b), start with the data $(Y,s_Y,F,\phi)$ of the second paragraph of the proof. After adding a pair of 2-handles if necessary, we can assume $Y$ contains an $S^2\times S^2$ connected summand disjoint from $F$. Take the double $Z=DY=\partial(I\times Y)=Y\cup_M \overline{Y}$ and set $n_M=b_2(Y)+7=b_+(DY)+7=b_-(DY)+7$. We wish to modify the pair $(DY,DF)$ in the complement of $M$, to control $e(\nu F,\phi)$ while the homology class $2[DF]$ changes to attain the same square and divisibility as $c_1(X)$. Since $X$ is simply connected, $H^2(X)$ has no torsion. Since $c_1(X)$ is {\em characteristic}, i.e., it reduces mod 2 to $w_2(X)$, it follows that $X$ is spin if and only if $m=\di c_1(X)$ is even. In this case, we can change $e(\nu F,\phi)$ by any multiple of $\frac12 m^2$ by adding a component to $F$ representing the class $\frac{m}{2}(1,k)\in H_2(S^2\times S^2)$ (using the obvious basis). Note that $m\ne0$ since the given bound on $b_+(X)$ rules out the K3 surface. The divisibility of the new $[F]$ will be exactly $\frac{m}{2}$ since this is its intersection number with $(0,1)\in H_2(S^2\times S^2)$. By suitably choosing $k$ and summing $Y$ with fewer than $\frac12 m^2$ copies of $S^2\times S^2$ disjointly from $F$, we can arrange $\theta$ to equal $\tilde{\Theta}(J,s,\phi)$. (Recall that these agree mod 4 by the previous paragraph.) Continue to denote the 4-manifold pair as $(Z,Y)$ (without changing $n_M$, and noting that $Z$ is no longer $DY$), and let $\hat{F}$ be the modified version of $DF$ in $Z$. We have $b_+(Z)=b_-(Z)\le n_M-8+\frac12 m^2$ and $\di[\hat F]=\frac{m}{2}$. (We can replace $n_M$ by $n_M-8$ in the theorem for spin $X$ excluding the K3 surface.) Since the intersection forms of $X$ and $Z$ are even and $c_1(X)$ has the form $m\beta$, the squares $c_1^2(X)$ and $(2[\hat{F}])^2$ have a common factor of $2m^2$. Thus, we can set them equal by adding to $\hat{F}$ a component with divisibility $\frac{m}{2}$ lying in the $S^2\times S^2$-summand disjoint from $Y$.

When $X$ is not spin, we apply a similar procedure, although $m$ is odd. In this case, we can change $e(\nu F,\phi)$ by any multiple of $2m^2$, using the class $m(1,k)\in H_2(S^2\times S^2)$. The resulting $F$ has divisibility $m$, and we need fewer than $2m^2$ $S^2\times S^2$-summands to set $\theta=\tilde{\Theta}(J,s,\phi)$, so $b_+(Z)=b_-(Z)\le n_M-8+2 m^2$. This time, we know that $c_1^2(X)$ and $(2[\hat{F}])^2$ have a common factor of $m^2$, but we can only change $(2[\hat{F}])^2$ by a multiple of $8m^2$ using the $S^2\times S^2$-summand outside $Y$. Thus, we can only arrange $(2[\hat{F}])^2$ to equal $c_1^2(X)+lm^2$ for some $l\in\Z$ that can be changed by any multiple of 8. We wish to modify the class $c=2[\hat{F}]$ outside $Y$ so that it has the same square and divisibility as $c_1(X)$. Let $\mu=m$ or $3m$, whichever has square 1 mod 16. (These both have square 1 mod 8 since they have the form $2r+1$.) We set $c^2=c_1^2(X)$ by summing $Z$ outside $Y$ with up to 8 copies of $\pm\C P^2$ and adding $\mu$ times each generator to $c$, then choosing $l$ suitably. Now $b_\pm(Z)\le n_M+2 m^2$.  Note that there are two ways to do this, depending on whether we choose the relevant $l$ to be an even or odd multiple of 8, and the resulting manifolds $Z$ have signatures differing by 8. We can arrange to use at least one $\pm\C P^2$-summand in this process, by using a canceling pair if necessary. Thus, neither $Z$ nor $X$ is spin in this case, but by construction, $c$ is characteristic. In both the spin and nonspin cases, we have now constructed a characteristic element $c\in H_2(Z)$ (which is $[2\hat{F}]$ in the spin case) with the same square and divisibility as $c_1(X)$. This is represented by a cycle whose intersection with $Y$ is twice our suitably modified $F$.

To complete the proof of (b), we use Freedman's classification of closed, simply connected, topological 4-manifolds \cite{F} to construct a homeomorphism between $X$ and a  manifold of the form $Z\# Z'$, where $Z'$ is a topological manifold with even intersection form. Since the square of any characteristic element is congruent mod 8 to the signature (e.g.~\cite[Lemma~1.2.20]{GS}), and $c^2=c_1^2(X)$, the signatures of $X$ and $Z$ are congruent mod 8. In the spin case, they are congruent mod 16 by Rohlin's theorem, and in the nonspin case we arrange this by choice of $Z$. By hypothesis and construction, we have $b_+(Z)\le b_+(X)$ and $b_-(Z)\le b_-(X)$. Since the intersection forms $Q(X)$ and $Q(Z)$ are indefinite, and both even when $X$ is spin and both odd otherwise, the classification of unimodular forms over $\Z$ gives an isomorphism $Q(X)\cong Q(Z)\oplus Q^\perp$, where $Q^\perp$ is some even form with signature divisible by 16. (First add a hyperbolic summand to $Q(Z)$ to set the smaller of $b_+(X)-b_+(Z)$ or $b_-(X)-b_-(Z)$ to 0, then add $\pm E_8$ summands.) By Freedman's classification, $Q^\perp$ is realized by some simply connected topological manifold $Z'$ with Kirby--Siebenmann invariant $\frac18\sigma(Z')\equiv 0$ mod 2. Since $X$ and $Z$ are smooth, their Kirby--Siebenmann invariants vanish. Thus, Freedman gives a homeomorphism from $Z\# Z'$ to $X$. By Wall \cite{W}, the automorphisms of $Q(X)$ act transitively on characteristic elements of a given divisibility and square (since $b_{\pm}(X)\ge2$). Thus, we can choose Freedman's homeomorphism to identify $c$ (which is still characteristic since $Q^\perp$ is even) with the Poincar\'e dual of $c_1(X)$. The homeomorphism restricts to a topological embedding $Y\hookrightarrow X$ that generates a Stein onion by Theorem~\ref{onion}. (In contrast to (a), we apply the theorem to an embedding that is typically unsmoothable.) Then $c_1(J_Y)=c_1(X)|Y$ is dual to twice our modified $F$, so by Theorem~\ref{Y}, the resulting boundary collar smoothing has the required invariants to exhibit TPC embeddings of $(M,J)$. If there is a spin structure on $X$, it restricts to the unique $s_Y$ on $Y$ and to $s$ on $M$.
\end{proof}

\begin{Remarks}\label{necHyps}
(a) The hypotheses of Theorem~\ref{main}(b) are necessary in some form. A bicollared topological embedding $M\hookrightarrow X$ into a closed, simply connected complex surface cuts out a compact 4-manifold $Y$. To see that the lower bounds on $b_\pm(X)$ in the theorem must depend on $\di c_1(X)$, suppose that for fixed $M$, $X$ ranges over elliptic surfaces in a fixed homotopy type. Then the Mayer-Vietoris sequence shows that the Betti numbers of $Y$ are bounded, so $|2\chi(Y)+3\sigma(Y)|$ is bounded by some $k$. For $\di c_1(X)$ sufficiently large and $\Gamma(J,s)=0$ (for example) not all values of $\tilde{\Theta}(J,s,\phi)\in\Z$ can be realized (for fixed $s$ and $\phi$) since any change in its first term is divisible by $(\di c_1(X))^2\gg k$. (This follows from Theorem~\ref{Y} for suitable hypotheses on $Y$, and from Definition~\ref{thetaDef} and Theorem~\ref{Jfunctor} in general.) Similarly, the bound on $\di c_1(X)$ in Theorem~\ref{main0}(a) is necessary when $c_1(J)$ has finite order.

To see the difficulty with divisibility by 4 in the final statement of Theorem~\ref{main}, suppose that $H^2(M)\cong\Z\oplus\Z/2$, and we wish to realize a pair $(J,s)$ for which $\Gamma(J,s)=(2n,1)$, by a bicollared topological immersion into a simply connected, spin complex surface $X$ with 4 a factor of $\di c_1(X)$, which is itself a factor of $4n$. Since $c_1(J)=(4n,0)$, this setup satisfies the other hypotheses of the theorem in both cases (a) (via Theorem~\ref{main0}(b)) and (b) (when  $b_\pm(X)$ is sufficiently large). However, such an immersion is impossible: Since $J$ and $s$ would pull back from structures $J_X$ and $s_X$ on $X$, Definition~\ref{defGamma} would provide a class $\Gamma(J_X,s_X)\in H^2(X)$ pulling back to $\Gamma(J,s)$ (via the immersed bicollar and Theorem~\ref{Jfunctor}), with $2\Gamma(J_X,s_X)=c_1(X)$ (Proposition~\ref{Gamma} on $X-\{p\}$). Since $c_1(X)$ has a factor of 4 and $H^2(X)$ has no 2-torsion, this would imply that $\Gamma(J_X,s_X)$ has a factor of 2, as does its pullback $\Gamma(J,s)$ to $M$. However, all elements in $H^2(M)$ with a factor of 2 have the form $(2k,0)$. Of course, the other spin structure $s'$ on $M$ giving the same Chern class works since $\Gamma(J,s')=(2n,0)$.

\item[(b)] The TPC embeddings constructed in this section all explicitly arise from embedded Stein onions. The author does not know any other general constructions. However, there is a simple trick for producing TPC embeddings that do not lie in Stein onions: Suppose $M$ bounds a 2-handlebody $Y$ that topologically embeds in $\C^2$. For example, we could take $M=P\#\overline{P}$ as in the introduction that has no SPC embedding in any complex surface. Apply Theorem~\ref{onion} to embed $Y$ in a Stein onion in $\C^2$. After translating and rescaling, we can assume the unit ball at 0 lies in $\inter Y$. By compactness, $Y$ lies in a ball at 0 with some radius $R$. Removing 0 from $\C^2$ and modding out by multiplication by $R$, we obtain a TPC embedding of $M$ in a Hopf surface $X$ diffeomorphic to $S^1\times S^3$. This embedded $M$ does not bound anything in $X$ (even immersed) since it represents a generator of $H_3(X)$. By Theorem~\ref{main0}(d), there is a similar TPC immersion of any 3-manifold with any $J$ such that $c_1(J)=0$.
\end{Remarks}

\section{The structure of $\J(V)$}\label{Conc}

We now prove the remaining theorems of Section~\ref{TPC}. For this, we need to understand the set $\J(V)$ of homotopy classes of almost-complex structures on a 4-manifold $V$ homotopy equivalent to a 3-manifold. We define a canonical $\Z$-action on $\J(V)$ and a complete set of invariants describing $\J(V)$ as a $\Z$-space. These are invariant under sliced concordance, hence, under topological isotopy when $V$ is smoothly embedded in another 4-manifold. They are concrete enough to use in completing the required proofs for Section~\ref{TPC}. They also further elucidate the $\Z$-spaces $\J(V)$, ultimately showing that these comprise a functor on suitable manifolds $V$ and homeomorphisms (not just diffeomorphisms) between them (Theorem~\ref{Jfunctor}).

\subsection{The invariant $\Gamma$}

We begin by defining and analyzing the invariant $\Gamma$ on $\J(V)$ that we will use in the next section to classify the orbits of the $\Z$-action. The first step is to understand the space of linear (orientation-preserving) complex structures on $\R^4$ with standard basis $e_0, e_1, e_2, e_3$. Up to homotopy, it suffices to consider orthogonal complex structures. The space of these is given by $\SO(4)/{\rm U}(2)=S^2$. In fact, any such $J$ is obtained from the standard complex structure $ie_0=e_1$, $ie_2=e_3$ by some rotation $R\in\SO(4)$. Precomposing with a unitary transformation preserves $J$, so we may assume $Re_0=e_0$. Then $Re_1$ must be $Je_0$. As $J$ varies, this can be any point on the unit sphere $S^2\subset\R^3={\rm span}\{e_1,e_2,e_3\}$. The vector $Je_0$ then completely determines $J$, since the latter must act on the orthogonal complement of $Je_0$ in $\R^3$ by a $+\frac{\pi}{2}$-rotation. Now that $S^2$ is identified with the set of orthogonal complex structures, we obtain a tautological complex bundle $E_{\rm{taut}}=S^2\times \R^4$, where the complex structure $J$ on the fiber over a given vector $v\in S^2$ is given by $Je_0=v$. There is a canonical complex identification $E_{\rm{taut}}=\C\oplus TS^2$, with first summand spanned over $\C$ by $e_0$.

Now let $E\to X$ be a trivial $\R^4$-bundle over a manifold (of any dimension). For any fixed trivialization $\tau$, the previous paragraph applies to each fiber, so the space of complex bundle structures on $E$ is homotopy equivalent to that of orthogonal complex structures (relative to $\tau$), and the latter structures correspond bijectively to maps $X\to S^2$. We obtain a bijection from the set $\J(E)$ of homotopy classes of complex bundle structures to the set $[X,S^2]$ of homotopy classes of maps, depending on $\tau$ only through its homotopy class. (It does depend crucially on the homotopy class of $\tau$; see Remark~\ref{tau}(a).) The Thom--Pontryagin construction (e.g.~\cite{M}) canonically identifies $[X,S^2]$ with the set $\Omega(X)$ of framed cobordism classes of (properly embedded) codimension-2 framed submanifolds of $\inter X$: Each $\varphi\co X\to S^2$ corresponds to the submanifold $\varphi^{-1}(p)$, where $p$ is any regular value, with normal framing chosen to map onto a fixed positive frame for $T_pS^2$. The inverse puts any framed submanifold into this form, where $\varphi$ is constant on the complement of a tubular neighborhood. It is routine to check that $\Omega$ is a contravariant functor: For each homotopy class of maps $f\co X\to Y$, define $f^*\co\Omega(Y)\to\Omega(X)$ by choosing $f$ transverse to a given framed submanifold and taking the preimage with pulled back framing. In fact, $\Omega$ is canonically isomorphic to the functor obtained from homotopy classes of maps to $S^2$ by composing functions in the obvious way. (The same discussion applies using codimension-$k$ submanifolds and $[X,S^k]$.)

To define $\Gamma$ and prove the subsequent proposition, we need a few basic facts about spin$^\C$-structures. On an $\R^4$-bundle, we can define complex, spin and spin$^\C$-structures to be respective lifts of the associated principal $\SO(4)$-bundle to the structure groups $\rm{U}(2)$, $\rm{Spin}(4)$ and $\rm{Spin}^\C(4)=(S^1\times \rm{Spin}(4))/\Delta$, where $\Delta$ is the diagonal $\Z/2$. Since the first two of these groups are canonically subgroups of the third, every complex or spin structure induces a spin$^\C$-structure. Since $\rm{Spin}^\C(4)$ canonically projects to $S^1\subset\C$, every spin$^\C$-structure determines a complex line bundle $L$ and hence a first Chern class. When the spin$^\C$-structure comes from a complex structure $J$, its Chern class $c_1(L)$ equals $c_1(J)$. When it comes from a spin structure $s$, its resulting Chern class satisfies $c_1(s)=-c_1(s)$, because $\rm{Spin}(4)\subset\rm{Spin}^\C(4)$ is the fixed set of the involution conjugating $S^1$ and $L$. When spin$^\C$-structures exist, they are classified by their primary difference obstruction in $H^2(X;\Z)$. Equivalently, this group acts freely and transitively on the set of spin$^\C$-structures (when the latter is nonempty). Similarly, $H^1(X;\Z/2)$ classifies spin structures, and the map from spin- to spin$^\C$-structures is equivariant under the action induced through the Bockstein homomorphism $H^1(X;\Z/2)\to H^2(X;\Z)$. Twice the difference class of a pair of spin$^\C$-structures is the difference of their Chern classes, so spin$^\C$-structures are determined by their Chern classes when $H^2(X)$ has no 2-torsion.

\begin{de}\label{defGamma}
Given a complex structure $J$ and a spin structure $s$ on an $\R^4$-bundle over $X$, let $\Gamma(J,s)\in H^2(X,\Z)$ denote the difference class of their induced spin$^\C$-structures.
\end{de}

\begin{prop}\label{Gamma}
For a fixed trivialization $\tau$ of $E\to X$ as above, let $s$ be its induced spin structure. Then for any complex structure $J\in\J(E)$, the class $\Gamma(J,s)$ is Poincar\'e dual to the properly embedded submanifold $F=\varphi^{-1}(p)\subset X$ that (when suitably framed) represents $J$ in $\Omega(X)$. In particular, the class dual to $F$ only depends on $\tau$ through $s$. There is a $J$-complex trivialization $\tau_J$ on $E|(X-F)$ agreeing with $s$ and representing twice the generator of $\pi_1({\rm U}(2))$ on each meridian of $F$. For fixed $J$, the invariants $\Gamma(J,s)$ for varying spin structures range over all classes with $2\Gamma(J,s)=c_1(J)$.
\end{prop}

\begin{proof}
After a homotopy making $J$ $\tau$-orthogonal, $\tau$ and $J$ together determine a map $\varphi\co X\to S^2$. By construction, this is covered by a complex bundle map $E\to E_{\rm{taut}}$, such that $\tau$ is the pullback of the constant real trivialization of $E_{\rm{taut}}=S^2\times\R^4$. In particular, $J=\varphi^*J_{\rm{taut}}$ and $s=\varphi^*s_{\rm{taut}}$, where $s_{\rm{taut}}$ is the constant (and in fact unique) spin structure on $E_{\rm{taut}}$. Since $E_{\rm{taut}}\cong\C\oplus TS^2$, we have $c_1(J_{\rm{taut}})=2[S^2]$. Since $H^2(S^2)$ has no 2-torsion, $c_1(s_{\rm{taut}})=-c_1(s_{\rm{taut}})=0$, and division by 2 is unique. Thus, $\Gamma(J_{\rm{taut}},s_{\rm{taut}})=\frac12(c_1(J_{\rm{taut}})-c_1(s_{\rm{taut}}))=[S^2]$. The latter is Poincar\'e dual to the point $p\in S^2$, so $\Gamma(J,s)=\varphi^*\Gamma(J_{\rm{taut}},s_{\rm{taut}})$ is dual to $\varphi^{-1}(p)$. Similarly, $2\Gamma(J,s)=2\varphi^*[S^2]=c_1(J)$. The trivialization $\tau_J$ comes from pulling back the trivialization on $E_{\rm{taut}}|(S^2-\{p\})$ obtained by identifying the base with $\C$. By the long exact coefficient sequence for $\Z\to\Z\to\Z/2$, the image of the above Bockstein consists of the elements of $H^2(X;\Z)$ with order at most 2. Thus, we can change $\Gamma(J,s)$ by any order-2 element by varying $s$, realizing all classes with $2\Gamma(J,s)=c_1(J)$.
\end{proof}

The invariant $\Gamma$ classifies $\J(E)$ whenever $X$ has the homotopy type of a 2-complex, so it classifies $\J(X)$ when $X$ is a spin 4-manifold with such a homotopy type. In this context, $\J(E)$ is essentially the set of spin$^\C$-structures and $\Gamma(\cdot,s)$ is a bijection. However, a 3-cell causes further complications that we address in the next two sections.

\subsection{Canonical $\Z$-actions} We  construct the desired $\Z$-action on $\J(V)$, along with some useful $\Z$-actions on related sets. We restrict the previous discussion to the case where $X=V$ is an open 4-manifold with a homotopy equivalence $f\co V\to M$ to a closed (connected, oriented) 3-manifold, and $E=TV$. In the case of primary interest, $V$ is a possibly exotic smoothing of $\R\times M$, with $f$ the obvious projection (up to homotopy), but the general case has the same required properties:

\begin{prop}\label{V}
The ends of $V$ comprise a canonically ordered set of two elements. The intersection pairing on $H_*(V;\F)$ vanishes for any field $\F$. There is a spin structure on $V$, and any such extends to a trivialization of $TV$ (nonuniquely). There is an induced bijection $f^*\co\Omega(M)\to\Omega(V)$.
\end{prop}

\begin{proof}
Recall that a manifold with $n$ ends has an exhaustion by compact, codimension-0 submanifolds whose complements each have $n$ components. These components have noncompact closures with connected boundaries. Each end then corresponds to a nested sequence of such components, whose boundaries we will call {\em cuts} of the end. If $V$ had more than two ends, we would have the contradiction $b_3(V)>1=b_3(M)$: Cuts of two ends would be linearly independent in $H_3(V)$, distinguished by pairing with elements of $H^3(V)$ dual to properly embedded lines emanating from a third end. To locate two distinct ends, let $\alpha\in H_3(V)$ be the generator mapping to $[M]\in H_3(M)$. Then $\alpha$ is dual to a compactly supported cohomology class. Since $H^1$ is classified by $K(\Z,1)=S^1$, there is a map $\psi\co V\to S^1$ that is constant outside some compact set, such that $\alpha$ is represented by $N=\psi^{-1}(\theta)$ for some regular value $\theta\in S^1$. Then $f|N$ has degree 1, so $(f|N)_*\co H_*(N;\F)\to H_*(M;\F)$ is surjective. A pair of classes in $H_*(V;\F)$ can be isomorphically pushed forward to $M$ and then pulled back to $N$, so they are represented by cycles in $N$. After pushing one of these in the positive normal direction, they will be disjoint. Thus, the intersection pairing on $H_*(V;\F)$ vanishes as required. In particular, each class in $H_1(V;\Z/2)$  has vanishing intersection number with $N$. The latter then separates $V$ into two regions (not necessarily connected unless $N$ is), with the region containing a given point $x$ determined by the mod 2 intersection number of $N$ with an arc from $x$ to a preassigned base point. For any regular value $p\in M$, $f^{-1}(p)\cdot N=p\cdot M=+1$. Thus, $V$ must have two distinct ends, one in each of the two regions of $V-N$. (Otherwise, $f^{-1}(p)$ could be modified to create a closed 1-cycle intersecting $N$ nontrivially.) Each cut of an end has intersection number $\pm1$ with  $f^{-1}(p)$ (since the latter has vanishing intersection number with the boundary of the cobordism between the cut and $N$). If we orient the cuts so that the intersection number is $+1$ (equivalently, so that they represent $\alpha$), there is a unique {\em positive} end for which the cuts are positively oriented boundary components of the compact regions defining the ends.

To complete the proof, note that the Wu formula $\langle w_2(V),\beta\rangle=\beta\cdot\beta$ for all $\beta\in H_2(V;\Z/2)$ (e.g.~\cite{GS}) implies $w_2(V)=0$, so $V$ admits a spin structure. Since $\pi_2(\SO(4))=0$, every spin structure extends to a trivialization over the 3-skeleton of $V$ (nonuniquely since $\pi_3(\SO(4))\ne0$). Since $V$ has no cohomology above dimension 3, we obtain an extension to all of $V$. The bijection $f^*\co\Omega(M)\to\Omega(V)$ follows immediately from functoriality of $\Omega$.
\end{proof}

We will say a connected, oriented 3-manifold $N\subset V$ with $f_*[N]=[M]$ {\em cuts} $V$. A cut separates $V$ into two components, each containing an end of $V$.

\begin{prop}\label{action}
There is a canonical $\Z$-action on $\J(V)$. Its orbit space is sent bijectively to $H^2(V)$ by $\Gamma$ for any fixed spin structure $s$. A sliced concordance between two such manifolds $V_0$ and $V_1$ determines an isomorphism of the corresponding $\Z$-spaces $\J(V_i)$ that preserves $\Gamma$.
\end{prop}

\begin{proof}
We first construct several canonical $\Z$-actions, which will be of continued use. On $\Omega(M)$, such an action is generated by adding to any framed link a $+1$-framed unknot in a disjoint 3-ball. It is easy to see that this is well-defined on $\Omega(M)$, with inverse obtained from a $-1$-framed unknot, and it is equivalent to adding a right twist to the framing of one component of the link. The Thom--Pontryagin isomorphism now pulls this action back to $[M,S^2]$. When $M=S^3$, the resulting generator sends the constant map to the Hopf fibration. Thus, for general $M$, the generator modifies a given map $\varphi\co M\to S^2$ by first pinching the domain to $M\vee S^3$, then applying $\varphi$ to the first factor and the Hopf map to the second. The homotopy equivalence $f\co V\to M$ pulls back these actions to $\Omega(V)$ and $[V,S^2]$. The generator of the former changes a framed surface $F$, along an arbitrarily chosen line $L$ disjoint from $F$ and connecting the ends of $V$, by adding a properly embedded $\R\times S^1$ along $L$ with +1-framed normal bundle. This pulls back by the homotopy inverse of $f$ to the generator on $\Omega(M)$, so defines a $\Z$-action on $\Omega(V)$ (whose definition is independent of $(M,f)$). The generator on $[V,S^2]$ is similarly defined near $L$, first homotoping a given $\varphi\co V\to S^2$ to be constant near $L$, then splicing in a map that projects to the normal 3-disk of $L$ and applies the Hopf map to it. More generally, any homotopy equivalence $g\co W^n\to M$ induces a $\Z$-action on $[W,S^2]\cong[M,S^2]$, respecting homotopy equivalences of pairs $(W,g)$. (Every map $W\to S^2$ factors through $g$ after homotopy. Then replace $L$ in the construction by $g^{-1}(q)$ for a regular value $q\in M$.) The set $\J(V)$ is identified with $[V,S^2]$ by choosing a homotopy class of trivializations $\tau$ of $TV$. Given an almost-complex structure $J$, homotope $\tau$ to be a $J$-complex trivialization near $L$. Then the associated map to $S^2$ is constant near $L$, and splicing in the Hopf map as before corresponds to redefining $J$ by the generator on $\J(V)$. Since this only changes $J$ near $L$, the resulting $\Z$-action on $\J(V)$ does not depend on the homotopy class of $\tau$ (although the $\Z$-space isomorphism to $[V,S^2]$ does).

By Proposition~\ref{V}, $V$ admits a spin structure, and any such $s$ comes from a trivialization $\tau$. By Proposition~\ref{Gamma}, $\Gamma(\cdot,s)\co\J(V)\to H^2(V)$ corresponds, under the $\Z$-space isomorphism $\J(V)\cong\Omega(V)$ induced by $\tau$, to the canonical map $\eta$ sending each framed surface to its dual cohomology class. Since $H^2(V;\Z)$ is classified by $K(\Z,2)=\C P^\infty$, we can identify it with the set of cobordism classes of codimension-2 submanifolds of $V$. Then $\eta\co\Omega(V)\to H^2(V)$ simply forgets framings. It follows that $\eta$ is surjective since every class in $H^2(V)$ is represented by a surface with no compact components, whose normal bundle must then be trivial. Similarly, $\eta$ is constant on each $\Z$-orbit, since the generator adds a nullcobordant cylinder to each surface. Finally, $\eta$ is injective on the set of orbits. This is most easily seen by passing isomorphically to $\Omega(M)$. Any connected cobordism between two nonempty framed links in $M$ can be made into a framed cobordism after adding twists to one framing. Thus, classes sent to the same element of $H^2(V)$ lie in the same orbit in $\Omega(V)$. The second sentence of the proposition follows immediately.

A sliced concordance $W$ between $V_0$ and $V_1$ determines a homeomorphism between them, so identifies the groups $H^2(V_i)$ with each other and with $H^2(W)$. The bundles $TV_i$ extend to the trivial bundle $E\to W$ of tangent spaces to the slices (i.e., the kernel of the derivative of the given submersion to $I$). A complex structure $J_i$ on either $TV_i$ extends uniquely (up to homotopy) to a structure $J$ on $E$, determining a bijection between the sets $\J(V_i)$. This preserves the $\Z$-action, which a trivialization of $E$ identifies with that of $[W,S^2]$. Similarly, a spin structure $s_i$ extends uniquely to $s$ on $E$, identifying spin structures on $V_0$ with those on $V_1$.  Then $\Gamma(J,s)\in H^2(W)$ restricts to $\Gamma(J_i,s_i)$ on each $V_i$, so these correspond under the isomorphism between the spaces $H^2(V_i)$.
\end{proof}

\subsection{The invariant $\tilde{\Theta}$}

Recall that $\J(M)=\J(\R\times M)$ can be identified with the set of homotopy classes of plane fields on $M$, which was computed and applied in \cite{Ann}. In this special case, our present $\Gamma(J,s)$ equals the previous $\Gamma(\xi,s)$ \cite[Remark 2 following Corollary~4.9]{Ann}, although the present version is more general and perhaps more natural. We now complete the computation of $\Omega(M)\cong\J(M)$ as in \cite{Ann}, then define the remaining invariant $\tilde{\Theta}$ in our additional generality and use it to prove the remaining theorems of Section~\ref{TPC}.

\begin{prop}\label{orbit}
Each $\phi\in\Omega(M)$ lies in an orbit isomorphic to the $\Z$-space $\Z/2\di[\phi]$, where $[\phi]$ denotes the homology class of any 1-manifold in $M$ that when suitably framed represents $\phi$.
\end{prop}

\begin{proof}
We can represent the framed cobordism class $\phi$ by a framed knot $K\subset M$. By our previous proof, any other class in the same orbit as $\phi$ can then be represented by the same knot, but with some number $n\in\Z$ of twists added to its framing. Suppose there is a framed cobordism in $I\times M$ between these framed knots. Glue the boundary components together by the identity to obtain a closed surface in $S^1\times M$ whose homology class $\alpha\in H_2(S^1\times M)$ has self-intersection $\alpha^2=-n$. For $\kappa=[S^1\times K]$, the class $\alpha-\kappa$ is represented by a cycle disjoint from $M$, so $(\alpha-\kappa)^2=0=\kappa^2$. Thus, $-n=((\alpha-\kappa)+\kappa)^2=2(\alpha-\kappa)\cdot\kappa$. Since $[\phi]=[K]$ in $H_1(M)$, we have $\di[\phi]=\di\kappa$, so $2\di[\phi]$ must be a factor of $n$. Conversely, we can write $[K]=\di[\phi]\beta$ and find $\zeta\in H_2(M)$ with $\beta\cdot\zeta=-1$. We can then modify the cobordism $I\times K$ so that the resulting $\alpha$ is $\kappa+\zeta$, implying $n=-2\kappa\cdot\zeta=2\di[\phi]$. Thus, $2\di[\phi]$ times the generator of the $\Z$-action is trivial on the orbit of $\phi$.
\end{proof}

To summarize our understanding of $\J(V)$ at this point, a trivialization $\tau$ of $TV$ and the given homotopy equivalence $f\co V\to M$, respectively, determine $\Z$-space isomorphisms $\J(V)\cong\Omega(V)\cong\Omega(M)$ (proof of Proposition~\ref{action}). Fix $J\in\J(V)$, and let $s$ be the spin structure determined by $\tau$. By Proposition \ref{Gamma}, the corresponding class $\Gamma(J,s)\in H^2(V)$ is dual to a surface $F$ in $V$ which, when suitably framed, represents the class in $\Omega(V)$ corresponding to $J$. This, in turn, corresponds to a class $\phi\in\Omega(M)$ as in Proposition~\ref{orbit}. The orbit of $J\in\J(V)$ then has order $2\di [\phi]=2\di\Gamma(J,s)=\di c_1(J)$. This orbit corresponds to the set of classes in $\Omega(V)$ whose underlying surfaces are dual to $\Gamma(J,s)$, so it can be understood in $\Omega(M)$ as framings on a knot in $M$. The main difficulty with this setup is that, while the orbit in $\J(V)$ is identified with an orbit in $\Omega(V)$ that is picked out by $s$, the bijection between the two orbits depends on the entire trivialization $\tau$ (Remark~\ref{tau}(a)). Since it is difficult to keep track of such a trivialization, we define an invariant that equivariantly compares elements of the two orbits without reference to $\tau$ beyond its induced spin structure $s$. We wish to compare a given $J$ with a class $\phi\in\Omega(V)$ in the corresponding orbit. Represent $\phi$ by a framed surface $F\subset V$. Let $\tau_J$ be a $J$-complex tangent trivialization of the complement $V-{\rm nbd}\thinspace F$ of a tubular neighborhood of $F$, respecting $s$ and representing twice the generator of $\pi_1({\rm U}(2))$ on each meridian of $F$. Such a $\tau_J$ exists by Proposition~\ref{Gamma}, since $F$ is dual to $\Gamma(J,s)$ and the Thom--Pontryagin construction exhibits it in the required form $\varphi^{-1}(p)$. Cut $V$ along a 3-manifold $N$ separating its ends and transverse to $F$. Glue the positive side $V_N$ of $V$ to a compact manifold whose boundary is $N=\overline{\partial V_N}$, obtaining a 4-manifold $Y$ whose unique end is the positive end of $V$. We show below that this can be done so that $J$ extends to an almost-complex structure $J_Y$ on $Y$. Then the relative Chern class $c_1(J_Y,\tau_J)\in H^2(Y,V_N-{\rm nbd}\thinspace F)$ is dual to an embedded surface $\tilde{F}$ in $Y$ that intersects $V_N$ in two copies of $F\cap V_N$ inside ${\rm nbd}\thinspace F$, parallel in the framing inherited from $\phi$. Any choice of this $\tilde{F}$ gives a normal Euler number $e(\nu\tilde{F},\phi)$ relative to the given framing on $F$.

\begin{de}\label{thetaDef}
For $J\in\J(V)$, a spin structure $s$ on $V$, and any $\phi\in\Omega(V)$ dual to $\Gamma(J,s)$, choose  $\tau_J$, $Y$ and $\tilde{F}$ as above and let $\tilde{\Theta}(J,s,\phi)=e(\nu\tilde{F},\phi)-2\chi(Y)-3\sigma(Y)$ in $\Z/4\di c_1(J)$.
\end{de}

\begin{thm}\label{class}
The invariant $\tilde{\Theta}$ is well-defined and invariant under sliced concordance (where $J$, $s$ and $\phi$ transform in the obvious way). It is also invariant under $\Z$ acting simultaneously on $\J(V)$ and $\Omega(V)$, with the positive generator on $\J(V)$ alone decreasing $\tilde{\Theta}$ by 4 and on $\Omega(V)$ alone increasing it by 4. The invariants $\Gamma$ and $\tilde{\Theta}$ together classify $\J(V)$, with the bijection $\Gamma(\cdot,s)$ onto $H^2(V)$ classifying the orbits and $\tilde{\Theta}(\cdot,s,\phi)$ identifying the orbit of each $J$ with an index-4 coset of $\Z/4\di c_1(J)$.
\end{thm}

\noindent We identify the cosets precisely in Proposition~\ref{mod4} when $V$ is homeomorphic to $\R\times M$. The dependence of the invariants on $s$ is given in Proposition~\ref{vary}.

\begin{proof}
We show the almost-complex extension $Y$ exists as in \cite[Lemma~4.4]{Ann}. Let $Y^*$ be a compact, spin 2-handlebody with boundary $N$ and no 1-handles. Then $TY^*$ has a trivialization, identifying orthogonal complex structures over subsets of $Y$ with maps to $S^2$. Thus, $J|N$ can be extended over the cocores of the 2-handles and then over $Y^*-\{y\}$ for some $y\in\inter Y$. The final obstruction then lies in $[S^3,S^2]\cong\Z$. For a closed 4-manifold $X$, this obstruction is given by the integer $\frac14(c_1^2(J_X)-2\chi(X)-3\sigma(X))$ for $c_1(J_X)\in H^2(X-\{x\})\cong H^2(X)$. This can be killed by suitably summing with copies of $S^2\times S^2$ (cf.~proof of Theorem~\ref{main}(b)). Thus, we can modify $Y^*$ so that $J$ extends as required.

Now for fixed $J$, $s$, $\phi$, $F$ and $\tau_J$ on $V$, suppose we have two extensions $Y_i$ as above. Since $H_*(V)\cong H_*(M)$ is finitely generated, so is each $H_*(Y_i)$. Choose a new cut $N_+$ in $V$ close enough to the positive end that each compact region $Y_i^c\subset Y_i$ cut out by $N_+$ contains surfaces representing a basis for $H_2(Y_i;\Q)$. Then inclusion induces an epimorphism $H_2(Y_i^c;\Q)\to H_2(Y_i;\Q)$ whose kernel is represented by surfaces pairing trivially with every element of $H_2(Y_i^c;\Q)$ (since the image surfaces are nullhomologous). Thus, $\sigma(Y_i^c)=\sigma(Y_i)$. As in the previous paragraph, cap each $Y_i^c$ by the same compact manifold $Z$ with boundary $\overline{N_+}$, obtaining closed, almost-complex manifolds $X_i$. By Novikov additivity, $\sigma(X_i)=\sigma(Y_i^c)+\sigma(Z)=\sigma(Y_i)+\sigma(Z)$. Similarly, $c_1^2(X_i)= e(\nu\tilde{F}_i,\phi)+e(\nu\tilde{F}_Z,\phi)$, where the surfaces are defined analogously to $\tilde{F}$. Finally, $\chi(X_i)=\chi(Y_i^c)+\chi(Z)=\chi(Y_i)-\chi(V_{N_+})+\chi(Z)$. (Additivity of $\chi$ follows from the rational Mayer-Vietoris sequence since the homologies are finitely generated and the intersection has $\chi(N_+)=0$). Thus, if $\theta_i$ denotes $\tilde{\Theta}(J,s,\phi)$ computed using $Y_i$, we have $0=c_1^2(X_i)-2\chi(X_i)-3\sigma(X_i)=\theta_i+\Delta$, where $\Delta$ is independent of $i$. Hence, $\tilde{\Theta}(J,s,\phi)$ is independent of choice of extension $(Y,J_Y,\tilde{F})$.

To compare $\tau_J$ with another trivialization $\tau_J^+$ on $V-{\rm nbd}\thinspace F$ as above, put $\tau_J^+$ on $V_N-{\rm nbd}\thinspace F$ and $\tau_J$ elsewhere on $V-{\rm nbd}\thinspace F$ except on a collar $I\times N$. The obstruction to fitting these together is the relative Chern class $c=c_1(J,\tau_J,\tau_J^+)\in H^2(I\times N-{\rm nbd}\thinspace F,\ \partial I\times N-{\rm nbd}\thinspace F)$. Since both trivializations are compatible with $s$, $c$ reduces mod 2 to $w_2(I\times N-{\rm nbd}\thinspace F,\thinspace\tau_J,\tau_J^+)=0$. Thus, $c$ has a factor of 2. Since the two trivializations agree on meridians to $F$, $c$ evaluates trivially on the meridians. Hence, it is dual to a compact surface in $I\times N-{\rm nbd}\thinspace F$ whose boundary consists of meridians. This then extends to a closed surface $F_N$ in $I\times N$ with $[F_N]$ even. If we have computed $\tilde{\Theta}(J,s,\phi)$ as in the definition using $\tau_J$, the same data computes it relative to $\tau_J^+$, provided that we first splice $F_N$ into $\tilde{F}$. Since  $[F_N]$ is even, the method of Proposition~\ref{orbit} shows that $e(\nu\tilde{F},\phi)$ changes by a multiple of $4\di c_1(J)$, so $\tilde{\Theta}$ does not depend on the choice of $\tau_J$. (Ignoring $\tau_J$ gives a weaker invariant $\Theta_\phi(J)\in\Z/2\di c_1(J)$; see Remark~\ref{tau}(b).)

Given a sliced concordance $\pi\co W\to I$ between two manifolds $V_0$ and $V_1$ as above, we saw in proving Proposition~\ref{action} that structures $J_0$ and $s_0$ on $V_0$ uniquely determine $J_1$ and $s_1$ on $V_1$ (with $J_1$ determined up to homotopy and arbitrary within its class) through structures $J$ and $s$ on the bundle $E$ of tangents to the slices in $W$. Similarly, the homotopy equivalences $V_i\to W$, $i=0,1$, determine isomorphisms $\Omega(W)\cong\Omega(V_i)$ sending a given $\phi_0$ on $V_0$ to a class $\phi_1$ on $V_1$. Any framed surfaces $F_i\subset V_i$ representing $\phi_i$ will then comprise the boundary of a framed cobordism $(P,\phi)$ in $W$ (by the definition of the $\Omega$ functor when $W$ is not a smooth product). For $\phi_0$ dual to $\Gamma(J_0,s_0)$, $P$ is dual to $\Gamma(J,s)$, and Proposition~\ref{Gamma} gives a complex trivialization $\tau_J$ on $E|(W-{\rm nbd}\thinspace P)$. To prove invariance of $\tilde{\Theta}$, first suppose there is a compact 4-manifold $N$ cutting $W$, with $\pi|N$ a submersion to $I$. Then $\pi$ gives $N$ the form $I\times N_0$. Our construction cutting and capping $V_0$ to make $Y_0$ can now be extended by gluing a product with $I$ along $N$ to give a cobordism $Y$ from $Y_0$ to $Y_1$ with a submersion to $I$, and a complex structure $J_Y$ on the bundle of tangents to the levels that agrees with $J$ on the common part $W_N$ of their domains. Its Chern class $c_1(J_Y,\tau_J)$ is dual to a 3-manifold $\tilde{P}$ intersecting $W_N$ in two $\phi$-parallel copies of $P$, and intersecting each $Y_i$ in a suitable $\tilde{F}_i$. The normal Euler class $e(\nu \tilde{P},\phi)$ is dual to a compact 1-manifold $L\subset\tilde{P}$ whose intersection number with each $\tilde{F_i}\subset Y_i$ is $e(\nu \tilde{F_i},\phi_i)$. Since the orientation numbers of $\partial L$ add to 0, it follows that $e(\nu \tilde{F_0},\phi_0)=e(\nu \tilde{F_1},\phi_1)$. When $(W,N)=I\times (V_0,N_0)$, all three terms of $\tilde{\Theta}(J_i,s_i,\phi_i)$ agree for $i=0,1$. This proves in general that $\tilde{\Theta}(J,s,\phi)$ does not depend on the choices of $J$ and $F$ within their classes, so it is well-defined. For a general sliced concordance, $N$ may not exist globally, but we can still construct it locally, in $\pi^{-1}(U)$ for some interval neighborhood $U$ of any given $t_0$ in $I$. (Flow some $N_{t_0}$ transversely to $V_{t_0}$ in $W$, and obtain $U$ by compactness of $N_{t_0}$.) We can also arrange a proper homotopy equivalence $(V_{N_{t_0}},N_{t_0})\to(V_{N_t},N_t)$ for each $t\in U$. (For each $a,b\in U$, define $f_a^b\co V_{N_a}\to V_{N_b}$ using the flow on a sufficiently large compact subset of $V_{N_a}$ and the topological product structure farther away. Then continuity in the parameters $a,b$ implies $f_{t_0}^t$ and $f_t^{t_0}$ are homotopy inverses.)  Now all three terms of $\tilde{\Theta}(J|V_t,s|V_t, \phi|V_t)$ are constant for $t\in U$, so it is constant on the connected set $I$. Hence, $\tilde{\Theta}$ is invariant under sliced concordance.

To complete the proof, let $J'$ and $\phi'$ be obtained from $J$ and $\phi$ by the positive generator of the $\Z$-action. Proposition~\ref{action} shows that $\Gamma(J',s)=\Gamma(J,s)$, and its proof gives explicit descriptions of $J'$ and $\phi'$. Both are obtained by modification near a line $L$ that we can assume intersects $N$ in a single point $x$. A framed surface $F'$ representing $\phi'$ is obtained from $F$ for $\phi$ by adding a framed tube intersecting $N$ in a $+1$-framed unknot. To compute $\tilde{\Theta}(J,s,\phi')$, we cap this tube in $Y$ to create a surface $R$ diffeomorphic to $\R^2$, with relative normal Euler number $+1$. We obtain $\tau_J'$ by modifying $\tau_J$ in a neighborhood of $L\cap V_N$ containing $R$, adding two twists around the meridian of $R$. Then $\tilde{\Theta}(J,s,\phi')-\tilde{\Theta}(J,s,\phi)$ is the normal Euler number relative to $\phi'$ of the surface $\tilde{R}$ obtained from two copies of $R$, $\phi'$-parallel near infinity, by smoothing their intersection point. This Euler number is $+4$ since the intersection form is quadratic. (To check this, cap $R$ to a closed surface by adding a $\phi'$-framed 2-handle to $Y$.) To compute $\tilde{\Theta}(J',s,\phi)$, fix $F$ and note that $J'=J$ outside a neighborhood of $L$ disjoint from $F$. Since $\pi_2({\rm U}(2))=0$, we can obtain $\tau_{J'}$ by modifying $\tau_J$ near $L$. Construct $Y'$ for $J'$ from $Y$ for $J$ by a suitable connected sum at $x$, as in the first paragraph of the proof. Since this construction is localized near $L$, the change in $\tilde{\Theta}$ is determined by a single example. We take $J'$ to be the standard complex structure on $\C^2-\{0\}\approx\R\times S^3$. This gives $\tilde{\Theta}(J',s,\phi)=-2$, using $Y=\C^2$ with its unique spin structure $s$ and $\phi$ the framed cobordism class of the empty set (so $\tilde{F}$ is empty). The plane field $\xi'=TS^3\cap J'TS^3$ on $S^3$ corresponding to $J'$ is orthogonal to the right-handed Hopf fibration. In the proof of \cite[Theorem~4.16]{Ann}, it is shown that the plane field $\xi$ corresponding to $J$ is obtained from $\xi'$ by a reflection of $S^3$. (For $S^3$ identified with the unit quaternions so that $J'$ acts on the left, right multiplication induces a trivialization $\tau$ of $TS^3$ in which $\xi'$ is constant. A calculation shows that $\tau$ identifies the mirror image $\xi$ of $\xi'\in\J(S^3)$ with a left-handed Hopf map in $[S^3,S^2]$, so the negative generator sends $\xi'$ to $\xi$ as required.) For any closed complex surface $X$, a convex ball $B\subset X$ has $\xi'$ induced on its boundary. This appears as $\xi$ in the boundary orientation of $X-\inter B$. Thus, we can use the latter to compute $\tilde{\Theta}(J,s,\phi)=+2$, since the invariant vanishes for $X$ but removing $\inter B$ decreases its Euler characteristic by 1. Hence, $\tilde{\Theta}(J',s,\phi)-\tilde{\Theta}(J,s,\phi)=-4$ as required. Since the orbit of $J$ has exactly $\di c_1(J)$ elements, $\tilde{\Theta}(\cdot,s,\phi)$ now injects the orbit onto an index-4 coset of  $\Z/4\di c_1(J)$.
\end{proof}

\begin{cor}\label{sliced} If $V_0$ and $V_1$ are sliced concordant smoothings of a 4-manifold homotopy equivalent to a closed, oriented 3-manifold, then the $\Z$-spaces $\J(V_0)$ and $\J(V_1)$ are canonically isomorphic.
\end{cor}

\begin{proof}
A sliced concordance induces an isomorphism by Proposition~\ref{action}. Different sliced concordances give the same isomorphism since the invariants agree. (Note that by definition, there is a fixed homeomorphism $V_0\to V_1$, so a distinguished homotopy equivalence in $[V_0,V_1]$. Thus, a given $s$ and $\phi$ on $V_0$ canonically determine corresponding structures on $V_1$, independently of choice of sliced concordance.)
\end{proof}

\begin{proof}[Proof of Theorems~\ref{sliced0}, \ref{Ann} and \ref{Y}.]
The first of these is immediate. The second follows from Propositions~\ref{action} and \ref{Gamma} and Theorem~\ref{class} (with $V=\R\times M$). For the third, we are given a compact, spin 4-manifold $(Y,s)$ with $\partial Y=M$ and an almost-complex structure $J_Y$ on another smoothing of $\inter Y$. Since $H^3(Y;\Z/2)=0$, the latter smoothing is sliced concordant to the original. Thus, the given boundary collar smoothing $V$ is sliced concordant to the standard smoothing on $\R\times M$, and $J_Y$ corresponds to a structure $J$ on $Y$, restricting canonically to $\R\times M$. We can now compute the invariants for $J_Y|V$ directly in $Y$. We are given a surface $(F,\partial F)\subset (Y,M)$, twice whose Poincar\'e dual is $c_1(J)$. Since $2\Gamma(J,s)=c_1(J)$ on $Y$ and $H^2(Y)$ has no 2-torsion, the dual of $F$ equals $\Gamma(J,s)$. Thus, the restriction to $M$ of $\Gamma(J,s)$ is dual to $\partial F$ as required. To compute $\tilde{\Theta}(J,s,\phi)$ for any framing $\phi$ on $\partial F$, apply its definition with $Y$ as given, $\tau_J$ defined on $Y-{\rm nbd}\thinspace F$, and $\tilde{F}$ made from two copies of $F$ that are $\phi$-parallel near $M$. Then $e(\nu\tilde{F},\phi)=4e(\nu F,\phi)$ since the intersection pairing is quadratic. The result follows.
\end{proof}

\begin{Remarks}\label{tau}(a)
It is now easy to see that for a trivial $\R^4$-bundle $E\to X$, our bijection $\J(E)\to[X,S^2]$ depends crucially on a choice of trivialization $\tau$. If we ignore $\tau$, we are effectively allowing automorphisms of $E$. A nowhere zero section of $E$ splits any complex bundle structure $J$ as $\C\oplus L$ for a line bundle $L$ determined by its Chern class $c_1(L)=c_1(J)$. Thus, the resulting equivalence classes are classified by the Chern class $c_1(J)$. When $E=TV$ as above, we lose information from $\Gamma$ whenever $H^2(V)$ has 2-torsion. We can recover this by fixing a spin structure, but still cannot distinguish almost-complex structures lying in a given orbit of the $\Z$-action. For example, we lose all information if $M$ is a homology sphere, although there is a $\Z$-space isomorphism $\J(V)\cong\Z$ in that case.

(b) Several variations of $\tilde{\Theta}$ from \cite{Ann} generalize to sliced concordance invariants in the current setting. The 2-fold quotient $\Theta_\phi(J)\in\Z/2\di c_1(J)$ of $\tilde{\Theta}(J,s,\phi)$ does not depend on a spin structure and is obtained more simply, by letting $\tilde{F}$ be any surface dual to $c_1(J_Y)$, framed near the end of $Y$ by $\phi\in\Omega(V)$ dual to $c_1(J)|V_N$ (without reference to $\tau_J$ or $\Gamma(J,s)$). When $c_1(J)$ vanishes in rational homology, all of the information of $\tilde{\Theta}(J,s,\phi)$ is captured by an invariant $\theta(J)\in\Q$ that is independent of $s$ and $\phi$. This is obtained by representing $c_1(J_Y)$ as a rational multiple of a closed surface. The orbit of $J$ then realizes a coset of $4\Z$ in $\Q$. When $V$ is homeomorphic to $\R\times M$, the cosets of these invariants are determined by simple data on $M$; see Proposition~\ref{mod4} and Remark~\ref{mod4Rem}. The invariants including $\Gamma$ can be algorithmically computed for contact structures exhibited as Stein boundaries. For plane fields on a 3-manifold $M$,  other relations and properties appear in \cite{Ann} and \cite[Section~11.3]{GS}. These properties generalize at least to the case of $V$ homeomorphic to $\R\times M$ (cf.~proof of Proposition~\ref{mod4}). The invariant $\theta$ is now well-known for plane fields and sometimes assumed to be equivalent to Pontryagin's secondary obstruction. However, there is additional information encoded in $\theta$ that may not be fully understood. For example, Giroux \cite{Gi} and Honda~\cite{H} independently classified tight contact structures on lens spaces, distinguishing them, respectively, by $\theta$ and $\Gamma$. Thus, $\theta$ has been used to compare plane fields with different values of the primary obstruction $\Gamma$ (but some with the same $c_1$), which is beyond the capability of a secondary obstruction.  Similarly, $\Theta_\phi$ relates structures with the same infinite order $c_1$ but different $\Gamma$.
\end{Remarks}

\section{Sliced concordance classes and the functor $\J$}\label{Ends}

Having completed the proof of our main theorem exhibiting TPC embeddings, we further study the $\Z$-space $\J(V)$ when $V$ is a smoothing of $\R\times M$. We determine the coset of $\Z/4\di c_1(J)$ comprising the image of an orbit under $\tilde{\Theta}$, using a generalization to $V$ of the Rohlin invariant for spin 3-manifolds. This allows us to exhibit $\J$ as a functor on such spaces $V$ and homeomorphisms between them. While $\J$ is obviously functorial for diffeomorphisms via the corresponding tangent bundle isomorphisms, unsmoothable homeomorphisms are more subtle. The Rohlin invariant also allows us to  study the relation between diffeomorphism and sliced concordance of smoothings of $\R\times M$, as well as topologically collared ends. We can then show that up to sliced concordance, every smoothing $V$of $\R\times M$ and $J\in\J(V)$ are realized holomorphically, by an embedding in a closed complex surface, although for the nonstandard concordance class, the embedding is quite different from our previous TPC examples. We leave some open questions and a conjecture.

\subsection{Dependence of the invariants on auxiliary data}

To exhibit $\J$ as homeomorphism-invariant by functoriality, we must understand how our invariants depend on $s$ and $\phi$. We first need some simpler functors.

\begin{thm}\label{spinc}\cite{spinc} There are functors $\s$ and $\s^\C$ that assign to each manifold its set of spin (resp.~spin$^\C$) structures, and assign to each (orientation-preserving) proper homotopy equivalence a bijection preserving the action of $H^1(\thinspace\cdot\thinspace,\Z/2)$ (resp.~ $H^2(\thinspace\cdot\thinspace,\Z)$) and the map from spin to spin$^\C$-structures. These extend the obvious functors on diffeomorphisms and sliced concordances. \qed
\end{thm}

\noindent For homeomorphisms, our main case of interest, the basic idea is that the map $\SO(n)\to{\rm STop}(n)$ is a $\pi_1$-isomorphism, so we can replace the former by the latter in constructing the structure groups. For proper homotopy equivalences in general, one needs to pass to classifying spaces and replace ${\rm BSTop}$ with the classifying space ${\rm BSG}$ of Spivak normal fibrations. The functor $\J$ is more difficult to deal with since it involves higher homotopy. However, the functor $\Omega$ from the previous section is immediately applicable in this section since each homeomorphism is homotopic to a smooth homotopy equivalence, uniquely up to smooth homotopy.

We can now describe how our invariants vary with the auxiliary data. Given a proper homotopy equivalence $V\to\R\times M$, $J\in\J(V)$ and $s_+,s_-\in\s(V)\cong\s(\R\times M)\cong\s(M)$, let $d(s_+,s_-)\in H^1(V;\Z/2)\cong H^1(M;\Z/2)$ denote the difference class. Given $\phi_\pm\in\Omega(V)$ dual to $\Gamma(J,s_\pm)$, respectively, choose corresponding surfaces $F_\pm\subset V$.

\begin{prop} \label{vary}
The difference $\Gamma(J,s_+)-\Gamma(J,s_-)$ is given by the Bockstein of $d(s_+,s_-)$. The difference $\tilde{\Theta}(J,s_+,\phi_+)-\tilde{\Theta}(J,s_-,\phi_-)$ is given by $e(\nu\hat F,\phi_\pm)$, where $\hat F\subset V$ is any surface that agrees near the $\pm$ end of $V$ with two $\phi_\pm$-parallel copies of $F_\pm$ and whose class $[\hat F]\in H_2(M;\Z/2)$ is dual to $d(s_+,s_-)$.
\end{prop}

\begin{proof}
The first sentence follows immediately from Definition~\ref{defGamma} and preceding text. For the rest, note that for $\tau^{\pm}_J$ associated to $F_\pm$ and $s_\pm$ as in Definition~\ref{thetaDef} near the $\pm$-end of $V$, $c_1(J,\tau^+_J,\tau^-_J)$ is represented by a surface $\hat F$ with the required properties. (Note that $\hat F$ is dual to $w_2(V,s_+,s_-)$ in the compactly supported $\Z_2$-cohomology of $V$. After the homotopy equivalence to $\R\times M$, this can be identified with $[\R]\times d(s_+,s_-)$, so $[\hat F]$ is dual to $d(s_+,s_-)$ in $H^1(M;\Z/2)$.) It suffices to use this $\hat F$, since any other $\hat F'$ as in the proposition differs from this $\hat F$ by an even (compact) homology class, so has the same Euler number as $\hat F$ mod $4\di c_1(J)$ as in the proof of Proposition~\ref{orbit}. Compute $\tilde{\Theta}(J,s_-,\phi_-)$ as usual, by capping $V$ at a cut near its negative end and constructing $\tilde{F}$ using $F_-$ and $\tau_J^-$. Use this same manifold $Y$ to compute $\tilde{\Theta}(J,s_+,\phi_+)$, by thinking of the cut as near its positive end. For this, we change $\tilde{F}$ between the cuts by splicing in $\hat F$, and extend to the positive end using $F_+$ and $\tau_J^+$. The normal Euler number increases by $e(\nu\hat F,\phi_\pm)$, and the other terms of $\tilde{\Theta}$ are unaffected.
\end{proof}

\subsection{Rohlin invariants, $\tilde{\Theta}$ mod 4, and functoriality of $\J$}

Next, we generalize the Rohlin invariant of spin 3-manifolds to 4-manifolds $V$ homotopy equivalent to 3-manifolds.

\begin{de}
Let $N$ be a 3-manifold cutting $V$ (so separating its ends). The {\em Rohlin invariant} of a spin structure $s$ on $V$ is that of $N$ with the restricted spin structure, $\mu(V,s)=\mu(N,s)\in\Z/16$, where the latter is the signature mod 16 of any compact, spin manifold with spin boundary $(N,s)$.
\end{de}

\begin{prop}\label{mu}
The invariant $\mu(V,s)$ is well-defined and invariant under both sliced concordance and diffeomorphism. Its residue mod 8 is homeomorphism invariant (where $s$ transforms as in Theorem~\ref{spinc}). It equals the mod 16 signature of any open spin manifold $Y$ made by cutting and capping the negative end of $V$ as for Definition~\ref{thetaDef}.
\end{prop}

\noindent We always assume homeomorphisms and diffeomorphisms among such manifolds $V$ preserve both the orientation and the order of the ends.

\begin{proof}
Cut $V$ by a pair of disjoint 3-manifolds $N_\pm$, with $N_+$ on the positive side of $N_-$. The region between these cuts has signature 0 since all intersection numbers in $V$ vanish. Thus, if we cap the ends by compact manifolds $Y_\pm$ to make a closed spin manifold $X$, we have $\sigma(X)=\sigma(Y_+)+\sigma(Y_-)\equiv 0$ mod 16 by Novikov additivity and Rohlin's Theorem. For a different choice of $N_-$, we can assume $N_+$ is on the positive side of both choices, and use the same $Y_+$ for both. Then any choices of $Y_-$ will have the same signature mod 16, showing $\mu(V,s)$ is well-defined. Given a sliced concordance, we can cut any level $V_t$ by a 3-manifold that also cuts nearby levels, so $\mu(V_t,s)$ is locally, hence globally, constant on $I$. A given diffeomorphism $\varphi\co V\to V'$ will send a given $s$ and $N_-$ to corresponding structures in $V'$, and we can use the same $Y_-$ to compute the invariant for both. If $\varphi$ is only a homeomorphism, $\varphi(N_-)$ is only topologically embedded, so the resulting $X$ is only a topological spin manifold. Rohlin's Theorem no longer applies, but the intersection pairing of $X$ is still even and unimodular, so $\sigma(X)$ is still divisible by 8. For $Y$ as given, the proof of Theorem~\ref{class} shows that $\sigma(Y)=\sigma(Y^c)$, where the latter manifold is a cap for $V$, and the latter signature mod 16 is $\mu(V,s)$ by definition.
\end{proof}

Now for $V$ homeomorphic to $\R\times M$, we can sharpen the classification of $\J(V)$ by computing the image of $\tilde{\Theta}$. This is entirely determined by the spin 3-manifold $(M,s)$.

\begin{prop}\label{mod4}
For $V$ homeomorphic to $\R\times M$ and $J$, $s$ and $\phi$ on $V$, we have the mod 4 congruence $\tilde{\Theta}(J,s,\phi)\equiv 2(1+b_1(M))-\mu(M,s)$.
\end{prop}

\begin{proof}
Cap $V$ with an almost-complex spin 2-handlebody without 1-handles (as in the proof of Theorem~\ref{class}). The resulting open manifold $Y$ can be compactified to a topological manifold $Y^*$ with boundary $M$, so $\sigma(Y)=\sigma(Y^*)\equiv\mu(M,s)$ mod~8. Also $Y^*$ is simply connected and $H^3(Y^*)$ is dual to $H_1(Y^*,M)$, so $b_1(Y)=0=b_3(Y)$. Represent $\Gamma(J_Y,s_Y)$ by a surface $F$ in $Y$ and choose $\tau_J$ on $Y-{\rm nbd\thinspace}F$ (cf.~proof of Theorem~\ref{Y} after Corollary~\ref{sliced}). Then mod~4, $$\tilde{\Theta}(J,s, \phi)=4e(\nu F,\phi)-2\chi(Y)-3\sigma(Y)\equiv -2(\chi(Y)+\sigma(Y))-\sigma(Y)\equiv 2(1+n(Y))-\mu(M,s)$$ where $n(Y)$ is the nullity of its intersection form. (Recall that $b_2(Y)=b_+(Y)+b_-(Y)+n(Y)$ and the first two terms cancel mod 2 with $\sigma(Y)$.) But $n(Y)=n(Y^*)=b_1(M)$ via Poincar\'e duality in $Y^*$ since $b_1(Y^*)=0$.
\end{proof}

\begin{Remark}\label{mod4Rem}
For $V$ as above, we can also determine the images of the invariants $\Theta_\phi$ and $\theta$ described in Remark~\ref{tau}(b). The first satisfies the above formula provided that the framing $\phi$ on a dual of $c_1(J)|V$ comes from a framing on $\Gamma(J,s)$ (which occurs on an orbit of $4\Z$ depending on $s$ in general). The second is defined when $\Gamma(J,s)$ has finite order, and differs from the above by 4 times the $\Q/\Z$-valued linking square of the dual of $\Gamma(J,s)$ in $H_1(M)$; see the end of \cite[Section~11.3]{GS}.
\end{Remark}

\begin{thm}\label{Jfunctor}
There is a unique functor $\J$ assigning $\J(V)$ to each $V$ homeomorphic to a manifold of the form $\R\times M$ (allowing $M$ to vary) and assigning to each homeomorphism a $\Z$-space isomorphism preserving $\Gamma$ and $\tilde{\Theta}$. These agree with the isomorphisms induced by sliced concordances (Corollary~\ref{sliced}).
\end{thm}

\begin{proof}
First fix a homeomorphism $h\co V'\to V$ and a spin structure $s$ on $V$. By Theorem~\ref{spinc}, $s$ pulls back to a spin structure $s'$ on $V'$ so that the induced spin$^\C$-structures correspond. Then for any $J\in \J(V)$, $h^*\Gamma(J,s)=\Gamma(J',s')$ whenever $J'$ is in the unique orbit in $\J(V')$ whose spin$^\C$-structures correspond to that of $J$. Any $\phi\in\Omega(V)$ dual to $\Gamma(J,s)$ pulls back to $\phi'=h^*\phi\in\Omega(V')$ dual to $\Gamma(J',s')$. By hypothesis, the underlying 3-manifolds $M$ and $M'$ become homeomorphic after product with $\R$, so $\mu(M',s')\equiv\mu(M,s)$ mod~8 by Proposition~\ref{mu}. Then by Proposition~\ref{mod4}, $\tilde{\Theta}(J',s',\phi')$ lies in the same coset as $\tilde{\Theta}(J,s,\phi)$. Thus, there is a unique bijection of these cosets for which $\tilde{\Theta}$ is preserved, and this is $\Z$-equivariant by Theorem~\ref{class}. The resulting $\Z$-space isomorphism $\J(V)\cong\J(V')$ is independent of $(s,\phi)$ by Proposition~\ref{vary}: Properly homotope $h$ to a smooth map $f$ transverse to the given surfaces $F_\pm$ and $\hat F$ in $V$. Then $F_\pm$ with their given framings pull back to representatives $F'_\pm$ of $\phi'_\pm$ in $V'$, and we can assume $\hat F$ pulls back to a suitable $\hat F'$ for applying the proposition to $V'$. A generic normal vector field to $\hat F$ extending $\phi_\pm$ pulls back to $\hat F'$. Since $f$ has degree 1, the signed sum of the zeroes is the same for both vector fields, so $e(\nu\hat F',\phi'_\pm)=e(\nu\hat F,\phi_\pm)$ as required. Ranging over all $h$ gives a functor. The last sentence of the theorem follows from the observation that a sliced concordance $W$ induces the same correspondence of $s$ and $\phi$ as its underlying homeomorphism. This is automatic for $s$ and follows for $\phi$ by reinterpreting $\Omega$ as $[\thinspace\cdot\thinspace,S^2]$.
\end{proof}

\subsection{Sliced concordance and diffeomorphism}

Any $V$ with the homotopy type of a 3-manifold has exactly two sliced concordance classes of smoothings since $H^3(V;\Z/2)\cong\Z/2$. If $V$ comes endowed with a preferred smoothing, such as the obvious smoothing when $V=\R\times M$, the class containing this smoothing is distinguished. It is natural to ask how distinguished this class is, for example:

\begin{ques}\label{products}
Is there a pair of 3-manifolds $M$ and $M'$ (possibly the same) with a homeomorphism $\R\times M\approx \R\times M'$ such that the induced map of sliced concordance classes does not preserve the distinguished classes?
\end{ques}

\noindent Equivalently, must the smoothings induced by two different topological product structures on $V$ be sliced concordant? Theorem~\ref{RxM}, following from the next proposition, gives a partial answer. (See also Remark~\ref{ks} for further discussion.)

\begin{prop}\label{muplus8}
If $V$ is homotopy equivalent to a 3-manifold $M$, and $V'$ is another smoothing of $V$ that is not sliced concordant to it, then for all spin structures $s$ on $V$ we have $\mu(V',s)=\mu(V,s)+8$. In particular, for any smoothing $V'$ of a fixed $\R\times M$, $\mu(V',s)$ equals $\mu(M,s)$ (resp.~$\mu(M,s)+8$) whenever $V'$ is (resp.~is not) sliced concordant to the given product smoothing.
\end{prop}

\noindent That is, the Kirby-Siebenmann uniqueness obstruction is given by $\frac18(\mu(V',s)-\mu(V,s))\in\Z/2$.

\begin{proof}
Freedman \cite{F} constructed a smooth manifold $V_P$ homeomorphic to $\R\times S^3$ but cut by a smoothly embedded Poincar\'e homology sphere $P$. Since $P$ is the boundary of the $E_8$-plumbing, $\mu(V_P,s_P)=8$, where $s_P$ is the unique spin structure on $V_P$ (and restricts to the unique structure on $P$). In each of $V_P$ and the given $V$, find a properly embedded  line connecting the two ends. Remove a tubular neighborhood of each line, and glue the two resulting manifolds along their $\R\times S^2$ boundaries. The resulting manifold $V^*$ is still homeomorphic to $V$. However, a cut $N$ for $V$ determines a cut $N\#P$ for $V^*$, so $\mu(V^*,s)=\mu(N\# P,s)=\mu(N,s)+8=\mu(V,s)+8$. (Any spin structure on $V$ extends uniquely to $V^*$ and agrees with the one induced by the homeomorphism to $V$.) By invariance of $\mu$, $V^*$ cannot be sliced concordant to $V$, so it is sliced concordant to the given $V'$. Since $\mu(\R\times M,s)=\mu(M,s)$, the result follows.
\end{proof}

\begin{proof}[Proof of Theorem~\ref{RxM}]
To prove the second sentence of the theorem, suppose $\R\times M$ has another product smoothing $V$, so there is a diffeomorphism $f\co V\to \R\times M'$. Then $f$ determines a homotopy equivalence $M\to M'$. If $M$ is hyperbolic or Haken, then every homotopy equivalence $M\to M'$ is homotopic to a diffeomorphism \cite{Mo}, \cite{Wa}. Thus, $f$ is homotopic to a diffeomorphism $\R\times M\to\R\times M'$. After postcomposing $f$ with the inverse of this diffeomorphism, we may assume that $M'=M$ and $f$ is homotopic to the identity. The latter implies $f$ preserves each spin structure $s$. Since it is a diffeomorphism from $V$ to the standard product structure, $\mu(V,s)=\mu(\R\times M,s)$. Thus, Proposition~\ref{muplus8} shows $V$ is sliced concordant to the standard product smoothing as required. Since any self-homeomorphism of $\R\times M$ induces a product smoothing, it must now fix the two sliced concordance classes. Any diffeomorphism between smoothings of $\R\times M$ is a self-homeomorphism, so diffeomorphic smoothings must be sliced concordant, completing the proof for $M$ hyperbolic or Haken. If $M$ is instead given to be a $\Z/2$-homology sphere, we may not have the above diffeomorphism of 3-manifolds. However, in this case there is a unique spin structure, and every diffeomorphism between smoothings preserves the Rohlin invariant for this structure, so we reach the same conclusion.
\end{proof}

This reasoning shows more generally that when there is no homotopy equivalence $h\co M\to M'$ with $\mu(M',h_*s)=\mu(M,s)+8$ for each spin structure $s$, $M$ and $M'$ cannot satisfy Question~\ref{products}, so if $M=M'$, diffeomorphic smoothings of $\R\times M$ must be sliced concordant. Of course, such an $h$ cannot be homotopic to a diffeomorphism. While our homeomorphisms always preserve the orientations and ends of our 4-manifolds, one could also use these tools to study homeomorphisms reversing the orientation on $V$ or switching the ends (preserving orientation on $V$, so reversing it on $M$). In either case, The Rohlin invariant reverses sign by definition. As a further example, we consider lens spaces, the simplest 3-manifolds not covered by Theorem~\ref{RxM}, and where nontrivial homotopy equivalences can arise. The lens space $L(p,q)$ is covered by the theorem, as a $\Z/2$-homology sphere, whenever $p$ is odd. The $p=0$ case $S^2\times S^1$ is easy since both spin structures extend to $D^3\times S^1$ so have $\mu=0$. For even $p>0$, expand $-\frac{p}{q}$ as a continued fraction with coefficients $(a_1,\dots,a_n)$ by repeatedly rounding to the nearest even integer and taking the negative reciprocal of the remainder. Since $p$ is even, it turns out \cite[Solution of Exercise~5.7.21(b)]{GS} that $n$ is odd and each $a_i$ is even (notably the last one), and the Rohlin invariants of the two spin structures differ by the sum of the odd-indexed coefficients, which \cite[Solution of Exercise~5.7.17(a)]{GS} is $p$ for the circle bundle $L(p,1)$. We conclude:

\begin{prop}\label{lens}
Diffeomorphic smoothings of $\R\times L(p,q)$ are sliced concordant, except possibly when $p$ is even and $\sum a_{2j+1}\equiv 8$ mod 16. In particular, the statement holds for all circle bundles over $S^2$ with Euler class not 8 mod 16. \qed
\end{prop}

\begin{proof}[Proof of Corollary~\ref{end}]
We are given a 4-manifold $X$ with an end homeomorphic to $\R\times M$, and wish to canonically map the set of smoothings of the end to $\Z/2$. By a smoothing of the end, we mean a smoothing of a neighborhood of it, which we consider up to sliced concordance after passing to a smaller neighborhood. We assign $0\in\Z/2$ to such a smoothing if it is diffeomorphic to a product. The only subtlety is that there may not be a single neighborhood of the end in which two such smoothings both appear as products. However, we can find an $\R\times M$ neighborhood for one smoothing. For another smoothing as $\R\times M'$, the hypotheses allow us to arrange $M'=M$ as before. The second product structure then gives a topological embedding of $M$ in $\R\times M$. If $M$ has more than one spin structure, the hypotheses allow us to assume this embedding is homotopic to the inclusion determined by the first product structure. Since the two copies of $M$ determine the Rohlin invariants of their respective smooth structures, the latter are sliced concordant as required.
\end{proof}

\subsection{Realizing sliced concordance classes by holomorphic embeddings}

Since a smoothing $V$ of $\R\times M$ is homeomorphically identified with it by definition, Theorem~\ref{Jfunctor} identifies $\J(V)$ with $\J(M)$ even when $V$ is not sliced concordant to $\R\times M$. Thus, it makes sense to ask about realizing $J\in\J(M)$ holomorphically on a preassigned exotic smoothing of $\R\times M$.

\begin{thm}\label{holom}
For every triple $(M,J,s)$, every smoothing on $\R\times M$ is sliced concordant to one for which $J$ is realized holomorphically by a smooth embedding into a closed, simply connected, complex surface, as a boundary collar of an embedded, compact, topological manifold $Y$ homotopy equivalent to a wedge of 2-spheres, with a spin structure on $Y$ extending $s$.

(a) For the sliced concordance class of the standard smoothing, $Y$ can be arranged to be a topologically embedded 2-handlebody without 1-handles, with the form $Y_\sigma$ in an embedded Stein onion exhibiting levels of $M$ as TPC embeddings.

(b) For the other class, no smoothing even arises as a boundary collar smoothing of a topologically embedded handlebody. If $M$ is hyperbolic, Haken or $S^3$, such a smoothing is not diffeomorphic to a neighborhood of the end of any smoothing of the interior of a handlebody (with connected boundary).
\end{thm}

\noindent See the text before Theorem~\ref{sliced0} for boundary collar smoothings and Definitions~\ref{TPCdef} and~\ref{oniondef} for TPC embeddings and Stein onions. A compact topological 4-manifold admits a handle structure if and only if it is smoothable, so $Y$ must be unsmoothable for the sliced concordance class discussed in (b) (although its interior inherits a smoothing from the embedding).

\begin{proof}
By Theorem~\ref{main}(b), $(M,J)$ has a TPC embedding in any simply connected, nonminimal complex surface $X$ with $b_{\pm}(X)$ sufficiently large. Its proof gives a suitable 2-handlebody $Y\subset X$, with the given $s$ extending over $Y$ since $m=1$, and a boundary collar smoothing sliced concordant to the standard smoothing on $\R\times M$. For the other sliced concordance class, return to the last paragraph of that proof. At that point, we had manifolds that we now denote $F^*\subset Y^*\subset Z^*$. Before constructing $Q^\perp$, create a pair $(Z,Y)$ from $(Z^*,Y^*)$ by summing with two copies of  Freedman's closed, simply connected, topological 4-manifold whose intersection form is $E_8$, with one summand in $Y^*-F^*$ and the other in $Z^*-Y^*$. The spin structure $s$ on $Y^*$ uniquely extends over $Y$. The new $Y$ is unsmoothable, but $Z$ still has vanishing Kirby--Siebenmann invariant in $\Z/2$. Continuing the proof as before yields a topological embedding $Y\subset X$, although Theorem~\ref{onion} no longer provides a Stein onion since $Y$ is not a 2-handlebody. But $Y$ still has the homotopy type of a wedge of 2-spheres by Whitehead's Theorem. The induced boundary collar smoothing $V$ now lies in the other sliced concordance class, since $\mu(V,s)=\mu(M,s)+8$ due to the extra $E_8$ added to the intersection form of $Y$. To identify the induced complex structure $J_V$ on $V$, choose tubular neighborhoods $F^*\subset U'\subset\cl U'\subset U\subset Y^*$ avoiding the connected sum region. Then the smoothing on $U$ induced from its embedding in $X$ is sliced concordant to the original. Since $F^*$ is dual to the generator of $H^2(U,U-U')$ (classified by $\C P^\infty$), it is now related by a cobordism in the sliced concordance over $U'$ to a surface $F$ in $U'$ that is smooth in $X$. Then $F$ is dual to $\Gamma(J_X,s)$ in $Y$, and its end inherits a framing from the original on $\partial F^*\subset Y^*$. The normal Euler numbers of $F$ and $F^*$ are equal (as in the proof of Theorem~\ref{class}). Thus, $J_V$ has the same invariants as $J$, as at the end of the proof of Theorem~\ref{Y} (following Corollary~\ref{sliced}). (Adding the $E_8$ summand to $Y^*$ changes two terms of $\tilde{\Theta}$ by a fixed multiple of 4, but we can still realize any value in the appropriate coset; cf.~proof of Theorem~\ref{main}.) This proves everything except (b).

To complete the proof, first note that for any compact, topological 4-manifold $Y$ with boundary $M$, any two smoothings of $\inter Y$ have sliced concordant ends: The smoothings may differ by some element of $H^3(\inter Y;\Z/2)$, but the restriction map to a neighborhood of the end vanishes. Thus, if $Y$ is smoothable, every boundary collar smoothing is sliced concordant to the standard smoothing on $\R\times M$. For the remaining sentence of (b), suppose $Y$ is smooth with boundary $M'$, but we are only given a smoothing $V$ of $\R\times M$ diffeomorphic to a neighborhood of the end of some smoothing of $\inter Y$. Then one end of $V$ is diffeomorphic to the end of the latter smoothing, so it is sliced concordant to some product structure as before. For $M$ as hypothesized, applying Corollary~\ref{end} to $V$ shows that end is sliced concordant to the standard product structure. Since $\R\times M$ is isotopic to a neighborhood of its end, $V$ is sliced concordant to the standard structure.
\end{proof}

\begin{Remark}\label{ks} Without the hypotheses in the last sentence of the theorem, any pair $M$, $M'$ as in Question~\ref{products} would immediately yield a counterexample in any smooth $Y$ with boundary $M'$. To better understand this hypothetical example, smoothly identify the end of $\inter Y$ with $\R\times M'$. The given homeomorphism then embeds $M$ in $Y$, cutting it into two compact pieces $Y^*$ and $K$ with $Y^*$ a deformation retract of $\inter Y$. Each piece has nontrivial Kirby--Siebenmann invariant. (For example, we can take $Y$ to be spin and apply Proposition~\ref{muplus8}, then notice that the invariant must be the same for both pieces, and $K$ is independent of choice of $Y$.) The nonstandard sliced concordance class on $\R\times M$ is then exhibited by a boundary collar smoothing, but for the unsmoothable manifold $Y^*$. The piece $K$ must be a 4-dimensional topological h-cobordism with nontrivial Kirby--Siebenmann invariant. The author does not know whether such a cobordism can actually exist.
\end{Remark}

The above theorem raises the question of whether a smoothing of $\R\times M$ induced by a TPC embedding of $M$ must be sliced concordant to the standard smoothing. This is always true for embeddings arising from Theorem~\ref{onion}, since these arise as boundaries of smoothable manifolds (topologically embedded 2-handlebodies) by construction. However, the definitions of Stein onion (in the present paper) and TPC embedding do not require smoothability of the embedded compact 4-manifolds. A basic example is Freedman's topological embedding of the Poincar\'e homology sphere $P$ in $\C^2$, as the boundary of a contractible manifold $\Delta$. This contractible manifold is unsmoothable, carrying the Kirby--Siebenmann obstruction in $H^4(\Delta,\partial\Delta;\Z/2)\cong\Z/2$, so the induced smoothing $V$ on $\R\times P$ is not sliced concordant to the product. (From our perspective, this is because $\mu(V,s)\equiv\sigma(\Delta)=0\equiv\mu(P,s)+8$ mod~16, where $s$ is the unique spin structure on each manifold.)

\begin{ques}\label{Delta}
Can $\Delta$ be assumed TPC? Can its interior be assumed Stein? Does $P$ have a TPC embedding bounding a homology ball in any complex surface?
\end{ques}

\noindent It is not clear to the author whether $\Delta$ admits a mapping cylinder structure (as required for the stronger Stein onion condition). We can at least write $\Delta$ as a nested intersection of smooth, compact 2-handlebodies in $\C^2$: Surround $\Delta$ by an arbitrarily small compact neighborhood with a smooth 3-manifold boundary. This neighborhood admits a handle structure without 4-handles. Since $\Delta$ is contractible, the cocores of the 3-handles can be homotoped, hence smoothly isotoped, off of it (since they are 1-dimensional). After deleting these cocores, we are left with a 2-handlebody containing a smaller neighborhood of $\Delta$. By applying Theorem~\ref{onion} to consecutive pairs, we can topologically isotope any finite collection of these 2-handlebodies to be Stein open subsets. However, we cannot expect the sequence of isotopies to converge to an isotopy on $\Delta$, since each such isotopy squeezes away most of the volume of the 2-handlebody. If $\inter \Delta$ is Stein, it is diffeomorphic to the interior of an infinite 2-handlebody. (Conversely, if it is homeomorphic to such a handlebody interior, it is topologically isotopic to a Stein surface \cite[Corollary~1.2]{steintop}.)

\begin{conj}
Every handlebody with interior homeomorphic to $\inter \Delta$ must have infinitely many 3-handles.
\end{conj}

\noindent Equivalently, this says that no smoothing of $\inter\Delta$ is diffeomorphic to the interior of a handlebody with at most finitely many 3-handles. The most complicated exotic smoothings of $\R^4$ (those with nonzero Taylor invariant) do not admit such diffeomorphisms \cite{T}, and one might expect smoothings of $\inter\Delta$ to be ``worse". However, some exotic smoothings of $\R^4$ that embed in $\C^2$ (as does $\inter \Delta$) admit such diffeomorphisms, and in fact, are Stein in the inherited complex structure \cite{steindiff}.

\end{document}